\documentclass[11pt]{amsart}
\usepackage{amssymb}
\usepackage{amsfonts}
\usepackage{amsmath}

\newtheorem{theorem}{Theorem}[section]
\newtheorem{proposition}{Proposition}[section]
\newtheorem{lemma}[theorem]{Lemma}
\newtheorem{corollary}[theorem]{Corollary}
\theoremstyle{definition}
\newtheorem{definition}[theorem]{Definition}

\theoremstyle{remark}

\newtheorem{acknowledgement}{Acknowledgement}
\numberwithin{equation}{section}

\begin{document}
\title[Quasilinear elliptic systems involving variable exponents]{Multiple
solutions for quasilinear elliptic systems involving variable exponents}

\begin{abstract}
We establish the existence of multiple solutions for a nonvariational
elliptic systems involving $p(x)$-Laplacian operator. The approach combines
the methods of sub-supersolution and Leray--Schauder topological degree.
\end{abstract}

\author{Abdelkrim Moussaoui}
\address{Abdelkrim Moussaoui\\
Applied Mathematics Laboratory (LMA), Faculty of Exact Sciences,\\
and Biology department, Faculty of natural and life sciences,\\
A. Mira Bejaia University, Algeria}
\email{abdelkrim.moussaoui@univ-bejaia.dz}
\author{Jean V\'{e}lin}
\address{Jean V\'{e}lin\\
D\'{e}partement de Math\'{e}matiques et Informatique, Laboratoire LAMIA,
Universit\'{e} des Antilles, Campus de Fouillole 97159 Pointe-\`{a}- Pitre,
Guadeloupe (FWI)}
\email{jean.velin@univ-antilles.fr}
\maketitle

\section{Introduction}

Let $\Omega $ be a bounded domain in $%
\mathbb{R}
^{N}$ ($N\geq 2$) with smooth boundary $\partial \Omega $. Given $p_{i}\in
C^{1}(\overline{\Omega }),$ $1<p_{i}^{-}\leq p_{i}^{+}<N$ with%
\begin{equation*}
\begin{array}{l}
p_{i}^{-}=\inf_{x\in \Omega }p_{i}(x)\text{ \ and \ }p_{i}^{+}=\sup_{x\in
\Omega }p_{i}(x),%
\end{array}%
\end{equation*}

we consider the quasilinear elliptic system%
\begin{equation*}
(\mathrm{P})\qquad \left\{ 
\begin{array}{ll}
-\Delta _{p_{1}(x)}u_{1}=f_{1}(x,u_{1},u_{2}) & \text{in }\Omega \\ 
-\Delta _{p_{2}(x)}u_{2}=f_{2}(x,u_{1},u_{2}) & \text{in }\Omega \\ 
u_{1},u_{2}=0 & \text{on }\partial \Omega ,%
\end{array}%
\right.
\end{equation*}%
where $\Delta _{p_{i}(x)}$ stands for $p_{i}(x)$-Laplacian differential
operator on $W_{0}^{1,p_{i}(x)}(\Omega )$ and the nonlinearities $%
f_{i}:\Omega \times \mathbb{R\times \mathbb{R}}\rightarrow \mathbb{R}$, $%
i=1,2$, are Carath\'{e}odory functions, i.e., $f_{i}(\cdot ,s_{1},s_{2})$ is
measurable for all $s_{1},s_{2}\in \mathbb{R}$ and $f_{i}(x,\cdot ,\cdot )$
is continuous for a.e. $x\in \Omega ,$ satisfying the following conditions:

\begin{description}
\item[$(\mathrm{H.1})$] $f_{1}$ and $f_{2}$ are bounded in bounded domain.

\item[$(\mathrm{H.2})$] There exists a constant $\eta _{i}>\lambda
_{1,p_{i}}\left\Vert \phi _{1,p_{i}}\right\Vert _{\infty }^{p_{i}^{+}-1}$
such that 
\begin{equation*}
\begin{array}{l}
\eta _{i}\leq \underset{s_{i}\rightarrow 0^{+}}{\liminf }\frac{%
f_{i}(x,s_{1},s_{2})}{s_{i}^{p_{i}^{-}-1}}%
\end{array}%
\end{equation*}%
$\ $ uniformly for a.e. $x\in \Omega $, all $s_{j}>0$, $i\not=j$, 
\begin{equation*}
\begin{array}{l}
\eta _{i}\leq \underset{s_{i}\rightarrow 0^{-}}{\liminf }\frac{%
f_{i}(x,s_{1},s_{2})}{|s_{i}|^{p_{i}^{-}-2}s_{i}}%
\end{array}%
\end{equation*}%
$\ $ uniformly for a.e. $x\in \Omega $, all $s_{j}<0$, $i\not=j$, $i=1,2$.
\end{description}

Here, $\lambda _{1,p_{i}}$ and $\phi _{1,p_{i}}$ denote the first eigenvalue
and the corresponding eigenfunction of $p_{i}(x)$-Laplacian operator,
respectively, for $i=1,2$.

\begin{description}
\item[$(\mathrm{H.3})$] 
\begin{equation*}
\begin{array}{l}
\underset{|s_{i}|\rightarrow \infty }{\lim }\sup \frac{f_{i}(x,s_{1},s_{2})}{%
|s_{i}|^{p_{i}^{-}-2}s_{i}}=0,%
\end{array}%
\end{equation*}%
uniformly for a.e. $x\in \Omega $, all $s_{j}\in 
\mathbb{R}
,$ $i,j=1,2,j\neq i,$ $i=1,2.$
\end{description}

\mathstrut

A solution $(u_{1},u_{2})\in W_{0}^{1,p_{1}(x)}(\Omega )\times
W_{0}^{1,p_{2}(x)}(\Omega )$ of problem $(\mathrm{P})$ is understood in the
weak sense, that is%
\begin{equation*}
\int_{\Omega }\left\vert \nabla u_{i}\right\vert ^{p_{i}(x)-2}\nabla
u_{i}\nabla \varphi _{i}\ \mathrm{d}x=\int_{\Omega
}f_{i}(x,u_{1},u_{2})\varphi _{i}\ \mathrm{d}x,
\end{equation*}%
for all $\varphi _{i}\in W_{0}^{1,p_{i}(x)}\left( \Omega \right) $.

Throughout this paper, we assume:

\begin{description}
\item[\textrm{(H}$_{p}$\textrm{)}] One of the following condition holds:

$\mathbf{(\mathrm{i})}$ There are two vectors $l_{i}\in \mathbb{R}%
^{N}\backslash \{0\}$ such that for all $x\in \Omega ,$ $%
h_{i}(t_{i})=p_{i}(x+t_{i}l_{i})$ are monotone for $t_{i}\in
I_{i,x}=\{t_{i};\,\,x+t_{i}l_{i}\in \Omega \},i=1,2.$

$\mathbf{(\mathrm{ii})}$ There is $x_{i}\notin \overline{\Omega }$ such that
for all $w_{i}\in \mathbb{R}\backslash \{0\}$ with $\Vert w_{i}\Vert =1,$
the function $h_{i}(t_{i})=p_{i}(x_{i}+t_{i}w_{i})$ is monotone for $%
t_{i}\in I_{x_{i},w_{i}}=\{t_{i}\in \mathbb{R};\,\,x_{i}+t_{i}w_{i}\in
\Omega \},$ for $i=1,2.$
\end{description}

Assumption $($\textrm{H}$_{p})$ ensures that Dirichlet problem%
\begin{equation}
-\Delta _{p_{i}(x)}u=\lambda |u|^{p_{i}(x)-2}u\text{ \ in }\Omega ,\text{ \ }%
u=0\text{ \ on }\partial \Omega ,  \label{26}
\end{equation}%
admits a first eigenvalue $\lambda _{1,p_{i}}>0$ caracterized by%
\begin{equation}
\lambda _{1,p_{i}}=\inf_{u\in W_{0}^{1,p_{i}(x)}(\Omega )\backslash \{0\}}%
\frac{\int_{\Omega }|\nabla u|^{p_{i}(x)}\,\mathrm{d}x}{\int_{\Omega
}|u|^{p_{i}(x)}\,\mathrm{d}x}  \label{71}
\end{equation}%
and the corresponding eigenfunction $\phi _{1,p_{i}}$ satisfies 
\begin{equation}
\phi _{1,p_{i}}\in C^{1}(\overline{\Omega }),\text{ }\phi _{1,p_{i}}>0\text{
in }\Omega \text{ and }\frac{\partial \phi _{1,p_{i}}}{\partial \nu }<0\text{
on }\partial \Omega  \label{27}
\end{equation}%
(see \cite{F, FZZ2}). Actually, assumption $($\textrm{H}$_{p})$ enables to
outfit $p_{i}(x)$-Laplacian operator with an important spectral property
that will be useful later on. However, this property alone does not make the
study of $(\mathrm{P})$ any easier because of the lack of properties such as
homogeneity. This fact complicates handling $p_{i}(x)$-Laplacian operator
and constitutes a serious technical difficulty to address problem $(\mathrm{P%
})$. Moreover, notice that system $(\mathrm{P})$ is not in variational form,
so the variational methods are not applicable.

Problems driven by the $p_{i}(x)$-Laplacian operator are involved in various
nonlinear processes related to electrorheological fluids \cite{AM1, R}, and
image restorations \cite{CLR}. When $p_{i}(\cdot )$ is reduced to be a
constant, $\Delta _{p_{i}(x)}$ becomes the well-known $p_{i}$-Laplacian
operator. In this context, system $(\mathrm{P})$ has been thoroughly
investigated in the litterature (see, e.g., \cite{DM, KM, MMP2, MM, MV1} and
the references therein). However, considering that $p(x)$-Laplacian operator
possesses more complicated nonlinearity, stretching out results of the
above-mentioned papers to problems involving $p(x)$-Laplacian operator is
not a straightforward task. This partly explains the few existing works in
the literature devoted to this topic. Actually, elliptic systems without
variational structure, possibly involving singularities near the origin, are
studied in \cite{AM, AMT, MV2, Z} while the variational case is considered
in \cite{MV3}. It should be noted that the systems considered in the
aforementioned papers do not fit the setting of $(\mathrm{P})$ under
assumptions $(\mathrm{H.1})$-$(\mathrm{H.3})$.

Surprisingly enough, excepting the quoted papers where existence of a
positive solution is obtained, so far we were not able to find previous
results providing more than one nontrivial solution for $(\mathrm{P})$.
Motivated by this fact, our main concern is the question of existence of
multiple solutions for a system of quasilinear elliptic equations $(\mathrm{P%
})$. We first establish the existence of opposite constant-sign solutions to
system $(\mathrm{P})$, which means the existence of a positive solution $%
(u_{1,+},u_{2,+})$ and a negative solution $(u_{1,-},u_{2,-})$ in the sense
that both components $u_{1,+},u_{2,+}$ are positive, and both components $%
u_{1,-},u_{2,-}$ are negative. Our approach is chiefly based on
sub-supersolutions method where a significant feature of our result lies in
the obtaining of the sub- and supersolutions for $(\mathrm{P})$. At this
point, the choice of suitable functions as well as an adjustment of adequate
constants is crucial. However, it is worth notting that the obtained sub-
and supersolution are quite different from the functions considered in the
quoted papers, especially those constructed in \cite{AM, AMT}. Practically
and contrary to preconceived ideas, the construction process of the sub- and
super-solutions in the present work is broadly similar to the one used in
the case of constant exponent problems (see, e.g., \cite{DKM, MM2, MM3}),
despite the loss of the homogeneity property of the operator $\Delta
_{p_{i}(x)},$ which constitutes in itself a major obstacle to face. The
crucial aspect of the argument is the new Mean Value Theorem (cf. Lemma 
\ref{L5}) which, henceforth, would become an essential tool to handle
problems with variable exponents.

The first main result is formulated as follows.

\begin{theorem}
\label{T1} Assume that conditions $(\mathrm{H.1})$, $(\mathrm{H.2})$ and $(%
\mathrm{H.3})$ hold. Then problem $(\mathrm{P})$ possesses at least a
positive solution $(u_{1,+},u_{2,+})$ and a negative solution $%
(u_{1,-},u_{2,-})$ in $C^{1,\sigma }(\overline{\Omega })\times C^{1,\sigma }(%
\overline{\Omega }),$ for certain $\sigma \in (0,1)$.
\end{theorem}

\mathstrut

Our next goal is to provide the existence of a second positive solution $(%
\breve{u}_{1},\breve{u}_{2})$ for system $(\mathrm{P})$. To this end, we
must strengthen hypothesis $(\mathrm{H.1})$ by the following assumption.

\begin{description}
\item[$(\mathrm{H}^{\prime }\mathrm{.1})$] 

$(\mathrm{i})$ $f_{i}(x,s_{1},s_{2})\geq 0$ uniformly for a.e. $x\in \Omega $%
, all $s_{i}\in 
\mathbb{R}
,$ $i=1,2.$

$(\mathrm{ii})$ For each $\delta >0,$ there exists $M=M(\delta )>0$ such
that 
\begin{equation*}
|f_{i}(x,s_{1},s_{2})|\leq M,\text{ for a.e. }x\in \Omega \text{, }%
|s_{i}|\leq \delta \text{, all }s_{j}\in 
\mathbb{R}
,i,j=1,2,\text{ }j\neq i.
\end{equation*}
\end{description}

The second main result is stated as follows.

\begin{theorem}
\label{T2} Assume that conditions $(\mathrm{H}^{\prime }\mathrm{.1}),$ $(%
\mathrm{H.2})$ and $(\mathrm{H.3})$ hold. Then problem $(\mathrm{P})$ admits
a solution $(\breve{u}_{1},\breve{u}_{2})$ in $W_{0}^{1,p_{1}(x)}(\Omega
)\times W_{0}^{1,p_{2}(x)}(\Omega )$ such that%
\begin{equation*}
\breve{u}_{1}\neq u_{1,+}\text{ \ and }\breve{u}_{2}\neq u_{2,+}.
\end{equation*}
\end{theorem}

The proof is based on topological degree theory with suitable truncation as
well as the Mean Value Theorem (cf. Lemma \ref{L5}). Precisely, we prove
that the degree on a ball $\mathcal{B}_{\tilde{R}}$ containing the obtained
solutions in Theorem \ref{T1} is equal to $1$ while the degree in a bigger
ball $\mathcal{B}_{R}\supset \mathcal{B}_{\tilde{R}},$ with $\tilde{R}<R,$
holding all potential solutions of $(\mathrm{P})$ is $0$. By the excision
property of Leray-Schauder degree, this leads to the existence of a solution
for $(\mathrm{P})$\ different from those obtained in Theorem \ref{T1}.

The rest of the paper is organized as follows. Section \ref{S2} contains
some technical and useful results; Section \ref{S4} deals with the existence
of opposite constant-sign solutions; Section \ref{S3} establishes the
existence of multiple positive solutions.

\section{Preliminaries and technical results}

\label{S2}

Let $L^{p_{i}(x)}(\Omega )$ be the generalized Lebesgue space that consists
of all measurable real-valued functions $u$ satisfying%
\begin{equation*}
\begin{array}{l}
\rho _{p_{i}(x)}(u)=\int_{\Omega }|u(x)|^{p_{i}(x)}dx<+\infty ,%
\end{array}%
\end{equation*}%
endowed with the Luxemburg norm%
\begin{equation*}
\begin{array}{l}
\left\Vert u\right\Vert _{p_{i}(x)}=\inf \{\tau >0:\rho _{p_{i}(x)}(\frac{u}{%
\tau })\leq 1\},\text{ }i=1,2.%
\end{array}%
\end{equation*}%
The variable exponent Sobolev space $W_{0}^{1,p_{i}(\cdot )}(\Omega )$ is
defined by%
\begin{equation*}
\begin{array}{l}
W_{0}^{1,p_{i}(x)}(\Omega )=\{u\in L^{p_{i}(x)}(\Omega ):|\nabla u|\in
L^{p_{i}(x)}(\Omega )\}.%
\end{array}%
\end{equation*}%
The norm $\left\Vert u\right\Vert =\left\Vert \nabla u\right\Vert
_{p_{i}(x)} $ makes $W_{0}^{1,p_{i}(x)}(\Omega )$ a Banach space. The
product space $W_{0}^{1,p_{1}(x)}(\Omega )\times W_{0}^{1,p_{2}(x)}(\Omega )$
is endowed with the norm $\left\Vert (u,v)\right\Vert =\left\Vert
u\right\Vert +\left\Vert v\right\Vert .$

In what follows, for any constant $C>0,$ we denote by $\mathcal{B}_{C}$ the
ball in $W_{0}^{1,p_{1}(x)}(\Omega )\times W_{0}^{1,p_{2}(x)}(\Omega )$
defined by%
\begin{equation*}
\mathcal{B}_{C}:=\left\{ (u_{1},u_{2})\in W_{0}^{1,p_{1}(x)}(\Omega )\times
W_{0}^{1,p_{2}(x)}(\Omega ):\,\Vert (u_{1},u_{2})\Vert <C\right\} .
\end{equation*}%
For any $r\in \mathbb{R}$, we denote $r^{+}:=\max \{r,0\}$ and $r^{-}:=\max
\{-r,0\}$.

\mathstrut

Next we formulate a serie of technical Lemmas which will be useful ater on.

\begin{lemma}
\label{L1}$(i)$ For any $u\in L^{p(x)}(\Omega )$ it holds%
\begin{equation*}
\begin{array}{l}
\left\Vert u\right\Vert _{p(x)}^{p^{-}}\leq \rho _{p(x)}(u)\leq \left\Vert
u\right\Vert _{p(x)}^{p^{+}}\text{ \ if \ }\left\Vert u\right\Vert _{p(x)}>1,%
\end{array}%
\end{equation*}%
\begin{equation*}
\begin{array}{l}
\left\Vert u\right\Vert _{p(x)}^{p^{+}}\leq \rho _{p(x)}(u)\leq \left\Vert
u\right\Vert _{p(x)}^{p^{-}}\text{ \ if \ }\left\Vert u\right\Vert
_{p(x)}\leq 1.%
\end{array}%
\end{equation*}

$(ii)$ For $u\in L^{p(x)}(\Omega )\backslash \{0\}$ we have 
\begin{equation}
\left\Vert u\right\Vert _{p(x)}=c\text{ \ if and only if }\rho _{p(x)}\left( 
\frac{u}{c}\right) =1.  \label{normro}
\end{equation}
\end{lemma}

\begin{definition}
Let $u,v\in W^{1,p(x)}(\Omega )$. We say that $-\Delta _{p(x)}u\leq -\Delta
_{p(x)}v$ if for all $\varphi \in W_{0}^{1,p(x)}(\Omega )$ with $\varphi
\geq 0$,%
\begin{equation*}
\int_{\Omega }|\nabla u|^{p(x)-2}\nabla u\nabla \varphi \text{ }\mathrm{d}%
x\leq \int_{\Omega }|\nabla v|^{p(x)-2}\nabla v\nabla \varphi \text{ }%
\mathrm{d}x.
\end{equation*}
\end{definition}

\begin{lemma}
\label{L6}Let $u,v\in W^{1,p(x)}(\Omega )$. If $-\Delta _{p(x)}u\leq -\Delta
_{p(x)}v$ and $u\leq v$ on $\partial \Omega $, then $u\leq v$ in $\Omega $.
\end{lemma}

The next Lemma is crucial in our approach, which establishes a result of the
Mean Value Theorem type.

\begin{lemma}
\label{L5} Let $h\in L^{p^{\prime }(x)}(\Omega )$ and let $\mathrm{k}\in
L^{\infty }(\Omega )$ be positive functions such that $\mathrm{k}(x)\in
(m,M) $ for a.e. $x\in \Omega $, for constants $m,M>0$. Let $u\in
W_{0}^{1,p(x)}(\Omega )$ be the solution of the Dirichlet problem 
\begin{equation}
-\Delta _{p(x)}u=h\text{ in }\Omega ,\text{ \ }u=0\text{ on }\partial \Omega
.
\end{equation}%
Then, for every $\varphi \in W_{0}^{1,p(x)}(\Omega )$ with $\varphi \geq 0$
in $\Omega $, there exists a constant $\mathrm{\hat{k}}\in (m,M),$ $\mathrm{%
\hat{k}}:=\mathrm{\hat{k}}(\varphi )$, such that 
\begin{equation*}
\int_{\Omega }\mathrm{k}(x){|\nabla u|^{p(x)-2}}\nabla u\nabla \varphi \text{
}dx=\mathrm{\hat{k}}\int_{\Omega }{|\nabla u|^{p(x)-2}}\nabla u\nabla
\varphi \text{ }dx.
\end{equation*}
\end{lemma}

\begin{proof}
From the identity (2) in \cite[Lemma in page 823]{B} we get%
\begin{equation}
\begin{array}{l}
\int_{\Omega }\mathrm{k}(x){|\nabla u|^{p(x)-2}}\nabla u\nabla \varphi \text{
}dx \\ 
=m\int_{\Omega }|\nabla u|^{p(x)-2}\nabla u\nabla \varphi \text{ }%
dx+\int_{m}^{M}\left( \int_{\Omega (y)}|\nabla u|^{p(x)-2}\nabla u\nabla
\varphi \text{ }dx\right) dy%
\end{array}
\label{C1}
\end{equation}%
while the identity (3) (also in \cite[Lemma in page 823]{B}) implies%
\begin{equation}
\begin{array}{l}
\int_{\Omega }\mathrm{k}(x){|\nabla u|^{p(x)-2}}\nabla u\nabla \varphi \text{
}dx \\ 
=M\int_{\Omega }|\nabla u|^{p(x)-2}\nabla u\nabla \varphi \text{ }%
dx-\int_{m}^{M}\left( \int_{\omega (y)}|\nabla u|^{p(x)-2}\nabla u\nabla
\varphi \text{ }dx\right) dy%
\end{array}
\label{C2}
\end{equation}%
where 
\begin{equation*}
\Omega (y)=\left\{ x\in \Omega ;\,\,\mathrm{k}(x)>y\right\} ,\text{ \ }%
\omega (y)=\left\{ x\in \Omega ;\,\,\mathrm{k}(x)\leq y\right\} ,
\end{equation*}%
for $y$ $\in \lbrack m,M].$ Denote by $\chi _{\omega (y)}$ the
characteristic function of the subset $\omega (y).$ Since $\nabla \chi
_{\omega (y)}(x)=0$ in $\Omega ,$ it follows that%
\begin{equation*}
\begin{array}{r}
\int_{\omega (y)}|\nabla u|^{p(x)-2}\nabla u\nabla \varphi \text{ }%
dx=\int_{\Omega }|\nabla u|^{p(x)-2}\left( \nabla u\nabla \varphi \right)
\chi _{\omega (y)}\text{ }dx \\ 
=\int_{\Omega }|\nabla u|^{p(x)-2}\nabla u\nabla \left( \varphi \chi
_{\omega (y)}\right) \text{ }dx.%
\end{array}%
\end{equation*}%
Hence, testing with $\varphi \cdot \chi _{\omega (y)}\in
W_{0}^{1,p(x)}(\Omega )$ we obtain%
\begin{equation*}
\begin{array}{l}
\langle -\Delta _{p(x)}u,\varphi \cdot \chi _{\omega (y)}\rangle \\ 
=\int_{\Omega }|\nabla u|^{p(x)-2}\nabla u\nabla \left( \varphi \cdot \chi
_{\omega (y)}\right) \text{ }dx-\int_{\partial \Omega }|\nabla
u|^{p(x)-2}\left( \nabla u\cdot \vec{n}\right) \left( \varphi \cdot \chi
_{\omega (y)}\right) \text{ }dx \\ 
=\int_{\Omega }|\nabla u|^{p(x)-2}\nabla u\nabla \left( \varphi \cdot \chi
_{\omega (y)}\right) \text{ }dx=\int_{\Omega }h\left( \varphi \cdot \chi
_{\omega (y)}\right) \text{ }dx.%
\end{array}%
\end{equation*}%
Bearing in mind that 
\begin{equation*}
\int_{\Omega }h\varphi \cdot \chi _{\omega (y)}dx>0\text{ for all }\varphi
\geq 0,
\end{equation*}%
we conclude that 
\begin{equation*}
\int_{\Omega }|\nabla u|^{p(x)-2}\nabla u\nabla \left( \varphi \cdot \chi
_{\omega (y)}\right) \text{ }dx>0
\end{equation*}%
which forces%
\begin{equation*}
\int_{\omega (y)}|\nabla u|^{p(x)-2}\nabla u\nabla \varphi \text{ }dx>0.
\end{equation*}%
By a quite similaire argument we get 
\begin{equation*}
\int_{\Omega (y)}|\nabla u|^{p(x)-2}\nabla u\nabla \varphi \text{ }dx>0.
\end{equation*}%
Thus, from (\ref{C1}) et (\ref{C2}) we derive that 
\begin{equation*}
m\int_{\Omega }|\nabla u|^{p(x)-2}\nabla u\nabla \varphi dx\leq \int_{\Omega
}\mathrm{k}(x){|\nabla u|^{p(x)-2}}\nabla u\nabla \varphi dx
\end{equation*}%
and 
\begin{equation*}
\int_{\Omega }\mathrm{k}(x){|\nabla u|^{p(x)-2}}\nabla u\nabla \varphi \text{
}dx\leq M\int_{\Omega }|\nabla u|^{p(x)-2}\nabla u\nabla \varphi \text{ }dx.
\end{equation*}%
Consequently, invoking the intermediate value theorem there exists a
constant $\mathrm{\hat{k}}\in (m,M)$, depending on $\varphi$, such that 
\begin{equation*}
\int_{\Omega }\mathrm{k}(x){|\nabla u|^{p(x)-2}}\nabla u\nabla \varphi \text{
}dx=\mathrm{\hat{k}}\int_{\Omega }{|\nabla u|^{p(x)-2}}\nabla u\nabla
\varphi \text{ }dx.
\end{equation*}%
This ends the proof.
\end{proof}

\begin{corollary}
\label{C}Let $h\in L^{\infty }(\Omega )$ a positive function in $\Omega $
and let $u\in W_{0}^{1,p(x)}(\Omega )$ be the solution of the Dirichlet
problem 
\begin{equation}
-\Delta _{p(x)}u=h(x)\text{ in }\Omega ,\text{ \ }u=0\text{ on }\partial
\Omega .
\end{equation}%
Then, for every $\varphi \in W_{0}^{1,p(x)}(\Omega )$ with $\varphi \geq 0$
in $\Omega ,$ there exists $x_{0}\in \Omega ,$ depending on $\varphi $, such
that 
\begin{equation*}
\int_{\Omega }C^{p(x)-1}{|\nabla u|^{p(x)-2}}\nabla u\nabla \varphi \text{ }%
dx=C^{p(x_{0})-1}\int_{\Omega }{|\nabla u|^{p(x)-2}}\nabla u\nabla \varphi 
\text{ }dx,
\end{equation*}%
for every constant $C>0.$
\end{corollary}

\begin{lemma}
\cite{R}\label{L4}Let $w_{1}\geq 0$ and $w_{2}>0$ be two nonconstant
differentiable functions in $\Omega .$ For all $x\in \Omega $ define 
\begin{equation}
\begin{array}{ll}
\mathcal{L}_{1}(w_{1},w_{2})= & |\nabla w_{1}|^{p(x)}+(p(x)-1)|\nabla
w_{2}|^{p(x)}(\frac{w_{1}}{w_{2}})^{p(x)} \\ 
& -p(x)|\nabla w_{2}|^{p(x)-2}\nabla w_{2}\nabla w_{1}(\frac{w_{1}}{w_{2}}%
)^{p(x)-1},%
\end{array}
\label{70}
\end{equation}%
\begin{equation}
\begin{array}{ll}
\mathcal{L}_{2}(w_{1},w_{2})= & |\nabla w_{1}|^{p(x)}-p(x)|\nabla
w_{2}|^{p(x)-2}\nabla w_{2}\nabla (\frac{w_{1}^{p(x)}}{w_{2}^{p(x)-1}}).%
\end{array}%
\end{equation}%
Then $\mathcal{L}_{1}(w_{1},w_{2})=\mathcal{L}_{2}(w_{1},w_{2})\geq 0$.
\end{lemma}

\begin{lemma}
\label{L9}Assume $(\mathrm{H}_{p})$\textrm{\ }holds true and let 
\begin{equation}
0<J<\lambda _{1,p}(p^{-}-1).  \label{3}
\end{equation}%
Then, the Dirichlet problem 
\begin{equation}
\left\{ 
\begin{array}{ll}
-\Delta _{p(x)}u=J(\frac{u^{+}}{\max \{1,\Vert u\Vert \}})^{p(x)-1}+\delta
\lambda _{1,p}\phi _{1,p}^{p(x)-1} & \text{in }\Omega \\ 
u=0\text{ } & \text{on }\partial \Omega%
\end{array}%
\right.  \label{23}
\end{equation}%
does not admit solutions $u\in W_{0}^{1,p(x)}(\Omega )$ for every $\delta >0$
small.
\end{lemma}

\begin{proof}
Arguing by contradiction, let $u\in W_{0}^{1,p(x)}(\Omega )$ be a solution
of (\ref{23}). According to \cite[Theorem $4.1$]{FZ}, $u$ is bounded in $%
L^{\infty }(\Omega )$ and therefore, owing to \cite[Theorem 1.1]{F}, $u$ is
bounded in $C^{1,\sigma }(\overline{\Omega })$ for a certain $\sigma \in
(0,1)$. Furthermore, by strong maximum principle in \cite{FZZ2} one can write%
\begin{equation}
u\geq \delta \phi _{1,p(x)}\text{ \ in }\Omega ,\text{ for }\delta >0\text{
small.}  \label{1}
\end{equation}%
Applying Picone's Identity in Lemma \ref{L4} to functions $u$ and $\phi
_{1,p}+\varepsilon $ for $\varepsilon >0,$ and by Lemma \ref{L5}, there is 
\textrm{\={k}}$\in (p^{-},p^{+})$ such that%
\begin{equation*}
\begin{array}{l}
0\leq \int_{\Omega }\mathcal{L}_{2}(u,\phi _{1,p}+\varepsilon )\text{ }%
\mathrm{d}x \\ 
=\int_{\Omega }|\nabla u|^{p(x)}\text{ }\mathrm{d}x-\int_{\Omega
}p(x)|\nabla \phi _{1,p}|^{p(x)-2}\nabla \phi _{1,p}\nabla (\frac{u^{p(x)}}{%
(\phi _{1,p}+\varepsilon )^{p(x)-1}})\text{ }\mathrm{d}x \\ 
=\int_{\Omega }|\nabla u|^{p(x)}\text{ }\mathrm{d}x-\mathrm{\bar{k}}%
\int_{\Omega }|\nabla \phi _{1,p}|^{p(x)-2}\nabla \phi _{1,p}\nabla (\frac{%
u^{p(x)}}{(\phi _{1,p}+\varepsilon )^{p(x)-1}})\text{ }\mathrm{d}x \\ 
=\int_{\Omega }|\nabla u|^{p(x)}\text{ }\mathrm{d}x-\lambda _{1,p}\mathrm{%
\bar{k}}\int_{\Omega }\left( \frac{\phi _{1,p}}{\phi _{1,p}+\varepsilon }%
\right) ^{p(x)-1}u^{p(x)}\text{ }\mathrm{d}x \\ 
\leq \int_{\Omega }|\nabla u|^{p(x)}\text{ }\mathrm{d}x-\lambda
_{1,p}p^{-}\int_{\Omega }\left( \frac{\phi _{1,p}}{\phi _{1,p}+\varepsilon }%
\right) ^{p(x)-1}u^{p(x)}\text{ }\mathrm{d}x\mathrm{.}%
\end{array}%
\end{equation*}%
Passing to the limit as $\varepsilon \rightarrow 0,$ by means of the
Lebesgue dominated convergence theorem, we obtain%
\begin{equation}
\begin{array}{l}
0\leq \int_{\Omega }|\nabla u|^{p(x)}\text{ }\mathrm{d}x-\lambda
_{1,p}p^{-}\int_{\Omega }u^{p(x)}\text{ }\mathrm{d}x.%
\end{array}
\label{56}
\end{equation}%
Acting with $u$ in (\ref{23}) and using (\ref{1}) lead to%
\begin{equation}
\begin{array}{l}
\int_{\Omega }|\nabla u|^{p(x)}\text{ }\mathrm{d}x=\int_{\Omega }(\frac{%
Ju^{p(x)}}{(\max \{1,\Vert u\Vert \})^{p(x)-1}}+\delta \lambda _{1,p}\phi
_{1,p}^{p(x)-1}u)\text{ }\mathrm{d}x \\ 
\leq \int_{\Omega }(Ju^{p(x)}+(\delta \lambda _{1,p}\phi _{1,p})^{p(x)-1}u)%
\text{ }\mathrm{d}x\leq \int_{\Omega }(J+\lambda _{1,p})u^{p(x)}\text{ }%
\mathrm{d}x.%
\end{array}
\label{57}
\end{equation}%
Gathering (\ref{56})-(\ref{57}) together we get%
\begin{equation*}
\begin{array}{l}
0\leq (J-\lambda _{1,p}(p^{-}-1))\int_{\Omega }u^{p(x)}\text{ }\mathrm{d}x<0,%
\end{array}%
\end{equation*}%
a contradiction due to (\ref{3}). Consequently, problem (\ref{23}) has no
solutions.
\end{proof}

\section{Proof of Theorem \protect\ref{T1}: Opposit constant-sign solutions}

\label{S4}

We establish the existence of two opposite constant-sign solutions to system 
$(\mathrm{P})$. Our approach relies on sub-supersolutions method (see \cite[%
Theorem 3.1]{AMT}). We recall that a sub-supersolution for $(\mathrm{P})$
consists of two pairs $(\underline{u}_{1},\overline{u}_{1}),(\underline{u}%
_{2},\overline{u}_{2})\in W_{0}^{1,p_{1}(x)}(\Omega )\times
W_{0}^{1,p_{2}(x)}(\Omega )$ such that there hold $\overline{u}_{i}\geq 
\underline{u}_{i}$ in $\Omega $, and%
\begin{equation*}
\int_{\Omega }\left\vert \nabla \underline{u}_{i}\right\vert
^{p_{i}(x)-2}\nabla \underline{u}_{i}\nabla \varphi _{i}\ \mathrm{d}%
x-\int_{\Omega }f_{i}(x,u_{1},u_{2})\varphi _{i}\ \mathrm{d}x\leq 0,
\end{equation*}%
\begin{equation*}
\int_{\Omega }\left\vert \nabla \overline{u}_{i}\right\vert
^{p_{i}(x)-2}\nabla \overline{u}_{i}\nabla \varphi _{i}\ \mathrm{d}%
x-\int_{\Omega }f_{i}(x,u_{1},u_{2})\varphi _{i}\ \mathrm{d}x\geq 0,
\end{equation*}%
for all $\varphi _{i}\in W_{0}^{1,p_{i}(x)}\left( \Omega \right) $ with $%
\varphi _{i}\geq 0$ a.e. in $\Omega $ and for all $u_{i}\in
W_{0}^{1,p_{i}(x)}\left( \Omega \right) $ satisfying $u_{i}\in \lbrack 
\underline{u}_{i},\overline{u}_{i}]$ a.e. in $\Omega $, for $i=1,2$.

\mathstrut

\noindent \textbf{Existence of supersolution:} \newline

Let $\tilde{\Omega}$ be a bounded domain in $%
\mathbb{R}
^{N}$ with smooth boundary $\partial \tilde{\Omega}$, such that $\overline{%
\Omega }\subset \tilde{\Omega}.$ We denote by $\tilde{\lambda}_{1,p_{i}}$
the first eigenvalue of $-\Delta _{p_{i}(x)}$ on $W_{0}^{1,p_{i}(x)}(\tilde{%
\Omega})$ and by $\tilde{\phi}_{1,p_{i}}$ the positive eigenfunction
corresponding to $\tilde{\lambda}_{1,p_{i}}$, that is%
\begin{equation}
-\Delta _{p_{i}(x)}\tilde{\phi}_{1,p_{i}}=\tilde{\lambda}_{1,p_{i}}\tilde{%
\phi}_{1,p_{i}}^{p_{i}(x)-1}\ \text{in }\tilde{\Omega},\ \ \tilde{\phi}%
_{1,p_{i}}=0\text{ on }\partial \tilde{\Omega}.  \label{29}
\end{equation}%
By the definition of $\tilde{\Omega}$ and the strong maximum principle,
there exists a constant $\tau >0$ sufficiently small such that 
\begin{equation}
\tilde{\phi}_{1,p_{i}}\left( x\right) >\tau \text{ in }\overline{\Omega }.
\label{28}
\end{equation}%
Moreover, one can find a constant $\bar{\eta}>0$ such that%
\begin{equation}
\bar{\eta}<\min_{i=1,2}\left\{ \frac{\tilde{\lambda}_{1,p_{i}}}{2}\tau
^{p_{i}^{+}-1}\left\Vert \tilde{\phi}_{1,p_{i}}\right\Vert _{\infty
}^{-(p_{i}^{-}-1)}\right\} .  \label{67}
\end{equation}%
For a constant $\varepsilon \in (0,1)$ set%
\begin{equation}
(\overline{u}_{1},\overline{u}_{2})=\varepsilon ^{-1}(\tilde{\phi}_{1,p_{1}},%
\tilde{\phi}_{1,p_{2}}).  \label{sub1}
\end{equation}%
It follows that%
\begin{equation}
\begin{array}{l}
\int_{\Omega }|\nabla \overline{u}_{i}|^{p_{i}(x)-2}\nabla \overline{u}%
_{i}\nabla \varphi _{i}\text{ }\mathrm{d}x=\int_{\Omega }\varepsilon
^{-(p_{i}(x)-1)}|\nabla \tilde{\phi}_{1,p_{i}}|^{p_{i}(x)-2}\nabla \tilde{%
\phi}_{1,p_{i}}\nabla \varphi _{i}\text{ }\mathrm{d}x.%
\end{array}
\label{39}
\end{equation}%
Using (\ref{28}) and Corollary \ref{C}, there is $\bar{x}_{i}\in \Omega $
such that%
\begin{equation}
\begin{array}{l}
\int_{\Omega }\varepsilon ^{-(p_{i}(x)-1)}|\nabla \tilde{\phi}%
_{1,p_{i}}|^{p_{i}(x)-2}\nabla \tilde{\phi}_{1,p_{i}}\nabla \varphi _{i}%
\text{ }\mathrm{d}x \\ 
=\varepsilon ^{-(p_{i}(\bar{x}_{i})-1)}\tilde{\lambda}_{1,p_{i}}\int_{\Omega
}\tilde{\phi}_{1,p_{i}}^{p_{i}(x)-1}\varphi _{i}\text{ }\mathrm{d}x \\ 
\geq \varepsilon ^{-(p_{i}^{-}-1)}\tilde{\lambda}_{1,p_{i}}\int_{\Omega }%
\tilde{\phi}_{1,p_{i}}^{p_{i}(x)-1}\varphi _{i}\text{ }\mathrm{d}x \\ 
=\varepsilon ^{-(p_{i}^{-}-1)}\tilde{\lambda}_{1,p_{i}}\int_{\Omega }\frac{1%
}{2}(\tilde{\phi}_{1,p_{i}}^{p_{i}(x)-1}+\tilde{\phi}%
_{1,p_{i}}^{p_{i}(x)-1})\varphi _{i}\text{ }\mathrm{d}x \\ 
\geq \varepsilon ^{-(p_{i}^{-}-1)}\tilde{\lambda}_{1,p_{i}}\int_{\Omega }%
\frac{1}{2}(\tau ^{p_{i}^{+}-1}+\tilde{\phi}_{1,p_{i}}^{p_{i}(x)-1})\varphi
_{i}\text{ }\mathrm{d}x,%
\end{array}
\label{31}
\end{equation}%
provided $\varepsilon >0$ small enough. Since, from (\ref{67}), we have%
\begin{equation*}
\frac{1}{2}\tilde{\lambda}_{1,p_{i}}\tilde{\phi}_{1,p_{i}}^{p_{i}(x)-1}\geq 
\frac{1}{2}\tilde{\lambda}_{1,p_{i}}\left\{ 
\begin{array}{ll}
\tilde{\phi}_{1,p_{i}}^{p_{i}^{-}-1}(x) & \text{if }\tilde{\phi}%
_{1,p_{i}}(x)\geq 1 \\ 
\tilde{\phi}_{1,p_{i}}^{p_{i}^{+}-1}(x) & \text{if }\tilde{\phi}%
_{1,p_{i}}(x)<1%
\end{array}%
\right. \geq \bar{\eta}\tilde{\phi}_{1,p_{i}}^{p_{i}^{-}-1}(x)\text{ in }%
\Omega ,
\end{equation*}%
then it follows that%
\begin{equation}
\begin{array}{l}
\int_{\Omega }\varepsilon ^{-(p_{i}^{-}-1)}\frac{1}{2}\tilde{\lambda}%
_{1,p_{i}}\tilde{\phi}_{1,p_{i}}^{p_{i}(x)-1}\varphi _{i}\text{ }\mathrm{d}%
x\geq \int_{\Omega }\varepsilon ^{-(p_{i}^{-}-1)}\bar{\eta}\tilde{\phi}%
_{1,p_{i}}^{p_{i}^{-}-1}\varphi _{i}\text{ }\mathrm{d}x \\ 
=\int_{\Omega }\bar{\eta}(\varepsilon ^{-1}\tilde{\phi}%
_{1,p_{i}})^{p_{i}^{-}-1}\varphi _{i}\text{ }\mathrm{d}x=\int_{\Omega }\bar{%
\eta}\overline{u}_{i}^{p_{i}^{-}-1}\varphi _{i}\text{ }\mathrm{d}x,%
\end{array}
\label{30}
\end{equation}%
for all $\varphi _{i}\in W_{0}^{1,p_{i}(x)}(\Omega )$ with $\varphi _{i}\geq
0$. On the other hand, assumption $(\mathrm{H.3})$ yields $\rho =\rho (\bar{%
\eta})>0$ such that 
\begin{equation*}
\frac{f_{i}(x,s_{1},s_{2})}{|s_{i}|^{p_{i}^{-}-2}s_{i}}\leq \bar{\eta},\text{
for a.e. }x\in \Omega \text{, for all }|s_{i}|>\rho ,s_{j}\in 
\mathbb{R}
,
\end{equation*}%
while assumption $(\mathrm{H.1})$ ensures the existence of a constant $%
c_{\rho }>0$ for which we have%
\begin{equation*}
|f_{i}(x,s_{1},s_{2})|\leq c_{\rho }\text{, for a.e. }x\in \Omega \text{,
for all }|s_{1}|,|s_{2}|\leq \rho ,i=1,2.
\end{equation*}%
Thus, it turns out that%
\begin{equation}
\begin{array}{l}
|f_{i}(x,s_{1},s_{2})|\leq c_{\rho }+\bar{\eta}|s_{i}|^{p_{i}^{-}-1},\text{\
for a.e. }x\in \Omega \text{, for all }s_{i}\in 
\mathbb{R}
.%
\end{array}
\label{32}
\end{equation}%
For $\varepsilon $ small one may assume that%
\begin{equation}
\begin{array}{l}
\varepsilon ^{-(p_{i}^{-}-1)}\frac{1}{2}\tilde{\lambda}_{1,p_{i}}\tau
^{p_{i}^{+}-1}\geq c_{\rho }.%
\end{array}
\label{45}
\end{equation}%
Then, gathering (\ref{39}) - (\ref{45}) together yields%
\begin{equation*}
\begin{array}{c}
\int_{\Omega }|\nabla \overline{u}_{1}|^{p_{1}(x)-2}\nabla \overline{u}%
_{1}\nabla \varphi _{1}\text{ }\mathrm{d}x\geq \int_{\Omega }(c_{\rho }+\bar{%
\eta}\overline{u}_{1}^{p_{1}^{-}-1})\varphi _{1}\text{ }\mathrm{d}x \\ 
\geq \int_{\Omega }f_{1}(x,\overline{u}_{1},s_{2})\varphi _{i}\text{ }%
\mathrm{d}x%
\end{array}%
\end{equation*}%
and%
\begin{equation*}
\begin{array}{c}
\int_{\Omega }|\nabla \overline{u}_{2}|^{p_{2}(x)-2}\nabla \overline{u}%
_{2}\nabla \varphi _{2}\text{ }\mathrm{d}x\geq \int_{\Omega }(c_{\rho }+\bar{%
\eta}\overline{u}_{2}^{p_{2}^{-}-1})\varphi _{2}\text{ }dx \\ 
\geq \int_{\Omega }f_{2}(x,s_{1},\overline{u}_{2})\varphi _{i}\text{ }%
\mathrm{d}x,%
\end{array}%
\end{equation*}%
for all $\varphi _{i}\in W_{0}^{1,p_{i}(x)}(\Omega )$ with $\varphi _{i}\geq
0$, for all $(s_{1},s_{2})\in \lbrack 0,\overline{u}_{1}]\times \lbrack 0,%
\overline{u}_{2}]$. This proves that $(\overline{u}_{1},\overline{u}_{2})$
is a supersolution for system $(\mathrm{P})$.

\mathstrut

\noindent \textbf{Existence of subsolution:} \newline

Next, we show that%
\begin{equation}
(\underline{u}_{1},\underline{u}_{2})=\varepsilon (\phi _{1,p_{1}},\phi
_{1,p_{2}})  \label{sub2}
\end{equation}%
is a subsolution for $(\mathrm{P})$ for $\varepsilon \in (0,1)$. We claim
that $\overline{u}_{i}\geq \underline{u}_{i}$ in $\overline{\Omega }$.
Indeed, from (\ref{26}), (\ref{27}) and Corollary \ref{C}, there is $\bar{x}%
_{i}\in \Omega $ such that%
\begin{equation}
\begin{array}{l}
\int_{\Omega }\varepsilon ^{p_{i}(x)-1}|\nabla \phi
_{1,p_{i}}|^{p_{i}(x)-2}\nabla \phi _{1,p_{i}}\nabla \varphi _{i}\text{ }%
\mathrm{d}x \\ 
=\varepsilon ^{p_{i}(\bar{x}_{i})-1}\lambda _{1,p_{i}}\int_{\Omega }\phi
_{1,p_{i}}^{p_{i}(x)-1}\varphi _{i}\text{ }\mathrm{d}x \\ 
\leq \varepsilon ^{p_{i}^{-}-1}\lambda _{1,p_{i}}\int_{\Omega }\phi
_{1,p_{i}}^{p_{i}(x)-1}\varphi _{i}\text{ }\mathrm{d}x,%
\end{array}
\label{31*}
\end{equation}%
for $\varepsilon >0$ sufficiently small, for all $\varphi _{i}\in
W_{0}^{1,p_{i}(x)}(\Omega )$ with $\varphi _{i}\geq 0$. Then, on account of (%
\ref{sub1}), (\ref{sub2}), (\ref{31*}) and the first equality in (\ref{31}),
it holds 
\begin{equation*}
\begin{array}{l}
\int_{\Omega }|\nabla \underline{u}_{i}|^{p_{i}(x)-2}\nabla \underline{u}%
_{i}\nabla \varphi _{i}\text{ }\mathrm{d}x\leq \int_{\Omega }|\nabla 
\overline{u}_{i}|^{p_{i}(x)-2}\nabla \overline{u}_{i}\nabla \varphi _{i}%
\text{ }\mathrm{d}x,%
\end{array}%
\end{equation*}%
for all $\varphi _{i}\in W_{0}^{1,p_{i}(x)}(\Omega )$ with $\varphi _{i}\geq
0$. This proves the claim.

In view of assumption $(\mathrm{H.2})$ there exists $\hat{\rho}=\hat{\rho}%
(\eta _{i})>0$ such that 
\begin{equation*}
\frac{f_{i}(x,s_{1},s_{2})}{s_{i}^{p_{i}^{-}-1}}\geq \eta _{i},\text{ for
a.e. }x\in \Omega \text{, for all }0<s_{i},s_{j}<\hat{\rho}.
\end{equation*}%
Thus%
\begin{equation}
f_{i}(x,s_{1},s_{2})\geq \eta _{i}s_{i}^{p_{i}^{-}-1},\text{\ for all }%
0<s_{1},s_{2}<\hat{\rho}.  \label{37}
\end{equation}%
For $\phi _{1,p_{i}}(x)>1,$ in view of $(\mathrm{H.2}),$ one has%
\begin{equation*}
\begin{array}{l}
\lambda _{1,p_{i}}\phi _{1,p_{i}}^{p_{i}(x)-1}(x)\leq \lambda _{1,p_{i}}\phi
_{1,p_{i}}^{p_{i}^{+}-1}(x)\leq \lambda _{1,p_{i}}\left\Vert \phi
_{1,p_{i}}\right\Vert _{\infty }^{p_{i}^{+}-1}\leq \eta _{i}\leq \eta
_{i}\phi _{1,p_{i}}^{p^{-}-1}(x)\text{ in }\Omega ,%
\end{array}%
\end{equation*}%
while, if $\phi _{1,p_{i}}(x)\leq 1,$ we have%
\begin{equation*}
\begin{array}{l}
\lambda _{1,p_{i}}\phi _{1,p_{i}}^{p_{i}(x)-1}(x)\leq \lambda _{1,p_{i}}\phi
_{1,p_{i}}^{p_{i}^{-}-1}(x)\leq \eta _{i}\phi _{1,p_{i}}^{p^{-}-1}(x)\text{
in }\Omega .%
\end{array}%
\end{equation*}%
Hence, it turns out that%
\begin{equation}
\begin{array}{c}
\varepsilon ^{p_{i}^{-}-1}\lambda _{1,p_{i}}\int_{\Omega }\phi
_{1,p_{i}}^{p_{i}(x)-1}\varphi _{i}\text{ }\mathrm{d}x\leq \varepsilon
^{p_{i}^{-}-1}\eta _{i}\int_{\Omega }\phi _{1,p_{i}}^{p^{-}-1}\varphi _{i}%
\text{ }\mathrm{d}x \\ 
=\eta _{i}\int_{\Omega }(\varepsilon \phi _{1,p_{i}})^{p_{i}^{-}-1}\varphi
_{i}\text{ }\mathrm{d}x,%
\end{array}
\label{38}
\end{equation}%
for all $\varphi _{i}\in W_{0}^{1,p_{i}(x)}(\Omega )$ with $\varphi _{i}\geq
0$. Then, assuming $\varepsilon >0$ so small that $\varepsilon \phi
_{1,p_{i}}(x)\leq \hat{\rho},$ $\forall x\in \Omega $, $i=1,2$, gathering (%
\ref{sub2}), (\ref{31*}), (\ref{37}) and (\ref{38}) together yield%
\begin{equation*}
\begin{array}{c}
\int_{\Omega }|\nabla \underline{u}_{1}|^{p_{1}(x)-2}\nabla \underline{u}%
_{1}\nabla \varphi _{1}\text{ }dx=\int_{\Omega }\varepsilon
^{p_{1}(x)-1}|\nabla \phi _{1,p_{1}}|^{p_{1}(x)-2}\nabla \phi
_{1,p_{1}}\nabla \varphi _{1}\text{ }dx \\ 
\leq \int_{\Omega }\eta _{1}\underline{u}_{1}^{p_{1}^{-}-1}\varphi _{1}\text{
}dx\leq \int_{\Omega }f_{1}(x,\underline{u}_{1},s_{2})\varphi _{1}\text{ }dx,%
\end{array}%
\end{equation*}%
and%
\begin{equation*}
\begin{array}{c}
\int_{\Omega }|\nabla \underline{u}_{2}|^{p_{2}(x)-2}\nabla \underline{u}%
_{2}\nabla \varphi _{2}\text{ }dx=\int_{\Omega }\varepsilon
^{p_{2}(x)-1}|\nabla \phi _{1,p_{2}}|^{p_{2}(x)-2}\nabla \phi
_{1,p_{2}}\nabla \varphi _{2}\text{ }\mathrm{d}x \\ 
\leq \int_{\Omega }\eta _{2}\underline{u}_{2}^{p_{2}^{-}-1}\varphi _{2}\text{
}dx\leq \int_{\Omega }f_{2}(x,s_{1},\underline{u}_{2})\varphi _{2}\text{ }dx,%
\end{array}%
\end{equation*}%
for all $\varphi _{i}\in W_{0}^{1,p_{i}(x)}(\Omega )$ with $\varphi _{i}\geq
0$, for all $(s_{1},s_{2})\in \lbrack \underline{u}_{1},\overline{u}%
_{1}]\times \lbrack \underline{u}_{2},\overline{u}_{2}]$, showing that $(%
\underline{u}_{1},\underline{u}_{2})$ is a subsolution for $(\mathrm{P})$.

\mathstrut

\noindent \textbf{Proof of Theorem \ref{T1}:} \newline

Now we are in position to apply \cite[Theorem 3.1]{AMT} which garantees the
existence of a positive solution $(u_{1,+},u_{2,+})$ satisfying $\underline{u%
}_{i}\leq u_{i,+}\leq \overline{u}_{i}.$ By an analogous approach as before,
on the basis of assumptions $(\mathrm{H.1})$, $(\mathrm{H.2})$ and $(\mathrm{%
H.3})$, we can show that the pair of functions $(-\overline{u}_{1},-%
\underline{u}_{1})$ and $(-\overline{u}_{2},-\underline{u}_{2})$ constitute
a pair of negative sub- and supersolution for problem $(\mathrm{P})$.
Consequently, we obtain a negative solution $(u_{1,-},u_{2,-})$ within $[-%
\overline{u}_{1},-\underline{u}_{1}]\times \lbrack -\overline{u}_{2},-%
\underline{u}_{2}]$. Furthermore, the nonlinear regularity theory up to the
boundary (see \cite[Theorem 1.2]{F}) implies that the solutions $%
(u_{1,+},u_{2,+})$ and $(u_{1,-},u_{2,-})$ belong to $C^{1,\sigma }(%
\overline{\Omega })\times C^{1,\sigma }(\overline{\Omega })$ for some $%
\sigma \in (0,1)$. This completes the proof.

\section{Proof of Theorem \protect\ref{T2}: Positive solutions}

\label{S3}

In this section we show that problem $(\mathrm{P})$\ admits a second
positive solution different from $(u_{1,+},u_{2,+})$. The proof is based on
topological degree theory. Precisely, we prove that the degree of an
operator corresponding to system $(\mathrm{P})$ is equal to $0$ on a ball $%
\mathcal{B}_{R}$, while the degree is $1$ in a smaller ball $\mathcal{B}_{%
\tilde{R}}\subset \mathcal{B}_{R},$ with $\tilde{R}<R$. By the excision
property of Leray-Schauder degree, we find a positive solution $(\breve{u}%
_{1},\breve{u}_{2})$ in $\mathcal{B}_{R}\backslash \overline{\mathcal{B}_{%
\hat{R}}}$ such that $\breve{u}_{1}\neq u_{1,+}$ \ and $\breve{u}_{2}\neq
u_{2,+}$.

\subsection{Topological degree on $\mathcal{B}_{R}$}

For every $t\in \lbrack 0,1]$, we consider the problem 
\begin{equation*}
(\mathrm{P}_{t})\qquad \left\{ 
\begin{array}{ll}
-\Delta _{p_{i}(x)}u_{i}=f_{i,t}(x,u_{1},u_{2}) & \text{in }\Omega \\ 
u_{i}=0 & \text{on }\partial \Omega ,%
\end{array}%
\right.
\end{equation*}%
with 
\begin{equation}
\begin{array}{l}
f_{1,t}(x,u_{1},u_{2})=tf_{i}(x,u_{1},u_{2})+(1-t)\left[ J_{i}\frac{%
(u_{i}^{+})^{p_{i}(x)-1}}{(\max \{1,\Vert u\Vert \})^{p_{1}(x)-1}}+\delta
\lambda _{1,p_{i}}\phi _{1,p_{i}}^{p_{i}(x)-1}\right] ,%
\end{array}
\label{12}
\end{equation}%
where $\delta >0$ is a small constant and%
\begin{equation}
0<J_{i}<\lambda _{1,p_{i}}\min \{1,p_{i}^{-}-1\},\text{ }i=1,2.  \label{43}
\end{equation}

With a constant $R>0$, let define the homotopy 
\begin{equation*}
\begin{array}{lll}
\mathcal{H}: & [0,1]\times \overline{\mathcal{B}}_{R} & \rightarrow
W^{-1,p_{1}^{\prime }(x)}(\Omega )\times W^{-1,p_{2}^{\prime }(x)}(\Omega )
\\ 
& (t,u_{1},u_{2}) & \rightarrow (\mathcal{H}_{1}(t,u_{1},u_{2}),\mathcal{H}%
_{2}(t,u_{1},u_{2}))%
\end{array}%
\end{equation*}%
where $\mathcal{H}_{i}$ are given by 
\begin{equation*}
\begin{array}{l}
\left\langle \mathcal{H}_{i}(t,u_{1},u_{2}),\varphi _{i}\right\rangle
=\int_{\Omega }|\nabla u_{i}|^{p_{i}(x)-2}\nabla u_{i}\nabla \varphi \,%
\mathrm{d}x-\int_{\Omega }f_{i,t}(x,u_{1},u_{2})\varphi _{i}\text{ }\mathrm{d%
}x,%
\end{array}%
\end{equation*}%
for $\varphi _{i}\in W_{0}^{1,p_{i}(x)}(\Omega )$ and $\overline{\mathcal{B}%
_{R}}$ is the closure of $\mathcal{B}_{R}$ in $W_{0}^{1,p_{1}(x)}(\Omega
)\times W_{0}^{1,p_{2}(x)}(\Omega )$ with%
\begin{equation*}
\mathcal{B}_{R}:=\left\{ (u_{1},u_{2})\in W_{0}^{1,p_{1}(x)}(\Omega )\times
W_{0}^{1,p_{2}(x)}(\Omega ):\,\Vert (u_{1},u_{2})\Vert <R\right\} .
\end{equation*}

\begin{lemma}
The homotopies $\mathcal{H}_{1}$ and $\mathcal{H}_{2}$ are continuous and
compact.
\end{lemma}

\begin{proof}
We prove only the continuity of $\mathcal{H}_{1}$ because that of $\mathcal{H%
}_{2}$ can be justified similarly. Let $(t_{n},u_{1,n},u_{2,n})\in \lbrack
0,1]\times \overline{\mathcal{B}}_{R}$ with 
\begin{equation}
(t_{n},u_{1,n},u_{2,n})\rightarrow (t,u_{1},u_{2})\text{ \ in }[0,1]\times
W_{0}^{1,p_{1}(x)}(\Omega )\times {W_{0}^{1,p_{2}(x)}(\Omega )}.  \label{41}
\end{equation}%
Passing to relabeled subsequences, there holds the convergence 
\begin{equation}
u_{i,n}\rightarrow u_{i}\text{ \ a.e. in }\Omega  \label{44}
\end{equation}%
and there exists a function $h_{i}\in L^{p_{i}(x)}(\Omega )$ such that%
\begin{equation}
|u_{i,n}(x)|\leq h_{i}(x)\text{ \ a.e. \ in }\Omega ,\text{ for }i=1,2.
\label{34}
\end{equation}%
Noticing that 
\begin{equation*}
\begin{array}{l}
t_{n}f_{1}(x,u_{1,n},u_{2,n})-tf_{1}(x,u_{1},u_{2}) \\ 
=(t_{n}-t)f_{1}(x,u_{1,n},u_{2,n})+t\left[
f_{1}(x,u_{1,n},u_{2,n})-f_{1}(x,u_{1},u_{2})\right] ,%
\end{array}%
\end{equation*}%
it suffices to prove that 
\begin{equation}
\left\{ f_{1,t_{n}}(x,u_{1,n},u_{2,n})\right\} \rightarrow \left\{
f_{1,t}(x,u_{1},u_{2})\right\} \text{\ in }L^{\frac{p_{1}(x)}{p_{1}(x)-1}%
}(\Omega ).  \label{35}
\end{equation}%
From (\ref{32}) we have that $f_{1}(x,u_{1,n},u_{2,n})\in
L^{p_{1}(x)/p_{1}(x)-1}(\Omega )$ while the fact that $f_{1}$ is a Carath%
\'{e}odory function implies 
\begin{equation*}
f_{1}(x,u_{1,n}(x),u_{2,n}(x))\rightarrow f_{1}(x,u_{1}(x),u_{2}(x))\text{ \
a.e. \ in }\Omega .
\end{equation*}%
Using (\ref{32}), (\ref{34})\textbf{\ }and the embedding $%
W_{0}^{1,p_{1}(x)}(\Omega )\hookrightarrow L^{p_{1}(x)}(\Omega ),$\textbf{\ }%
it follows that%
\begin{equation*}
\left\vert f_{1}(x,u_{1,n_{k}},u_{2,n_{k}})-f_{1}(x,u_{1},u_{2})\right\vert
^{\frac{p_{1}(x)}{p_{1}(x)-1}}\leq \left[ 2C_{p}+\overline{\eta _{i}}\left(
|h|^{p_{1}^{-}-1}+|u_{1}|^{p_{1}^{-}-1}\right) \right] ^{\frac{p_{1}(x)}{%
p_{1}(x)-1}}.
\end{equation*}%
Then, the dominated convergence result in \cite[Lemma 2.3.16]{DHHR} implies
that (\ref{35}) holds true.

The next step in the proof is to show that\textbf{\ }%
\begin{equation*}
\begin{array}{l}
(1-t_{n})\frac{(u_{1,n}^{+})^{p_{1}(x)-1}}{\left( \max \left\{ 1,\Vert
u_{1,n}\Vert \right\} \right) ^{p_{1}(x)-1}}\rightarrow (1-t)\frac{%
(u_{1}^{+})^{p_{1}(x)-1}}{\left( \max \left\{ 1,\Vert u_{1}\Vert \right\}
\right) ^{p_{1}(x)-1}}\text{ in }L^{\frac{p_{1}(x)}{p_{1}(x)-1}}(\Omega ).%
\end{array}%
\end{equation*}%
As above one can write 
\begin{equation}
\begin{array}{l}
(1-t_{n})\frac{(u_{1,n}^{+})^{p_{1}(x)-1}}{\left( \max \left\{ 1,\Vert
u_{1,n}\Vert \right\} \right) ^{p_{1}(x)-1}}-(1-t)\frac{%
(u_{1}^{+})^{p_{1}(x)-1}}{\left( \max \left\{ 1,\Vert u_{1}\Vert \right\}
\right) ^{p_{1}(x)-1}} \\ 
=(t-t_{n})\frac{(u_{1,n}^{+})^{p_{1}(x)-1}}{\left( \max \left\{ 1,\Vert
u_{1,n}\Vert \right\} \right) ^{p_{1}(x)-1}} \\ 
\text{ \ \ }+(1-t)\left( \frac{(u_{1,n}^{+})^{p_{1}(x)-1}}{\left( \max
\left\{ 1,\Vert u_{1,n}\Vert \right\} \right) ^{p_{1}(x)-1}}-\frac{%
(u_{1}^{+})^{p_{1}(x)-1}}{\left( \max \left\{ 1,\Vert u_{1}\Vert \right\}
\right) ^{p_{1}(x)-1}}\right)  \\ 
=(t-t_{n})\frac{(u_{1,n}^{+})^{p_{1}(x)-1}}{\left( \max \left\{ 1,\Vert
u_{1,n}\Vert \right\} \right) ^{p_{1}(x)-1}} \\ 
\text{ \ \ }+(1-t)(u_{1,n}^{+})^{p_{1}(x)-1}\left( \frac{1}{\left( \max
\left\{ 1,\Vert u_{1,n}\Vert \right\} \right) ^{p_{1}(x)-1}}-\frac{1}{\left(
\max \left\{ 1,\Vert u_{1}\Vert \right\} \right) ^{p_{1}(x)-1}}\right)  \\ 
\text{ \ \ }+\frac{1-t}{\left( \max \left\{ 1,\Vert u_{1}\Vert \right\}
\right) ^{p_{1}(x)-1}}\left(
(u_{1,n}^{+})^{p_{1}(x)-1}-(u_{1}^{+})^{p_{1}(x)-1}\right) 
\end{array}
\label{42}
\end{equation}%
The triangle inequalities%
\begin{equation*}
\Vert u_{1,n}\Vert \leq \Vert u_{1,n}-u_{1}\Vert +\Vert u_{1}\Vert \text{ \
and \ }\Vert u_{1}\Vert \leq \Vert u_{1,n}-u_{1}\Vert +\Vert u_{1,n}\Vert 
\end{equation*}%
ensure that $\Vert u_{1}\Vert >1$ (resp. $\leq 1$) whenever $\Vert
u_{1,n}\Vert >1$ (resp. $\leq 1$) and therefore, due to (\ref{41}), one has%
\begin{equation*}
\max \left\{ 1,\left\Vert u_{1,n}\right\Vert \right\} \rightarrow \Vert
u_{1}\Vert =\max \left\{ 1,\Vert u_{1}\Vert \right\} .
\end{equation*}%
Hence, for all $x\in \Omega ,$ we have 
\begin{equation*}
\frac{1}{\left( \max \left\{ 1,\Vert u_{1,n_{k}}\Vert \right\} \right)
^{p_{1}(x)-1}}-\frac{1}{\left( \max \left\{ 1,\Vert u_{1}\Vert \right\}
\right) ^{p_{1}(x)-1}}\rightarrow 0,
\end{equation*}%
which implies that 
\begin{equation*}
\begin{array}{l}
\left\vert \frac{1}{\left( \max \left\{ 1,\Vert u_{1,n}\Vert \right\}
\right) ^{p_{1}(\cdot )-1}}-\frac{1}{\left( \max \left\{ 1,\Vert u_{1}\Vert
\right\} \right) ^{p_{1}(\cdot )-1}}\right\vert
^{p_{1}(x)/p_{1}(x)-1}\rightarrow 0.%
\end{array}%
\end{equation*}%
Moreover, thanks to the estimate 
\begin{equation*}
\begin{array}{l}
\left\vert \frac{1}{\left( \max \left\{ 1,\Vert u_{1,n}\Vert \right\}
\right) ^{p_{1}(x)-1}}-\frac{1}{\left( \max \left\{ 1,\Vert u_{1}\Vert
\right\} \right) ^{p_{1}(x)-1}}\right\vert \leq 2,%
\end{array}%
\end{equation*}%
we conclude, from the dominated convergence theorem, that 
\begin{equation*}
\begin{array}{l}
\frac{1}{\left( \max \left\{ 1,\Vert u_{1,n}\Vert \right\} \right)
^{p_{1}(\cdot )-1}}\rightarrow \frac{1}{\left( \max \left\{ 1,\Vert
u_{1}\Vert \right\} \right) ^{p_{1}(\cdot )-1}}\text{ in }L^{\frac{p_{1}(x)}{%
p_{1}(x)-1}}(\Omega ).%
\end{array}%
\end{equation*}%
Now, we focus on the last term in (\ref{42}). Observe that%
\begin{equation}
\begin{array}{l}
(u_{1,n}^{+})^{p_{1}(x)-1}-(u_{1}^{+})^{p_{1}(x)-1}=\chi _{\{u_{1,n}\geq
0\}}|u_{1,n}|^{p_{1}(x)-1}-\chi _{\{u_{1}\geq 0\}}|u_{1}|^{p_{1}(x)-1} \\ 
=\left( \chi _{\{u_{1,n}\geq 0\}}-\chi _{\{u_{1}\geq 0\}}\right)
|u_{1}|^{p_{1}(x)-1}+\chi _{\{u_{1}\geq 0\}}\left(
|u_{1,n}|^{p_{1}(x)-1}-|u_{1}|^{p_{1}(x)-1}\right) .%
\end{array}
\label{48}
\end{equation}%
Due to (\ref{44}) and the estimate $|\chi _{\{u_{1,n}\geq 0\}}(x)-\chi
_{\{u_{1}\geq 0\}}(x)|\leq 2,$ it follows that%
\begin{equation}
\chi _{\{u_{1,n}\geq 0\}}-\chi _{\{u_{1}\geq 0\}}\rightarrow 0\text{ in }L^{%
\frac{p_{1}(x)}{p_{1}(x)-1}}(\Omega ).  \label{47}
\end{equation}%
Moreover, since by (\ref{44}) and (\ref{34}) we have 
\begin{equation*}
|u_{1,n}|^{p_{1}(x)-1}-|u_{1}|^{p_{1}(x)-1}\rightarrow 0\text{ \ a.e }x\in
\Omega 
\end{equation*}%
and%
\begin{equation*}
\left\vert |u_{1,n}|^{p_{1}(x)-1}-|u_{1}|^{p_{1}(x)-1}\right\vert \leq
h^{p_{1}(x)-1}+|u_{1}|^{p_{1}(x)-1},
\end{equation*}%
where the positive function $h^{p_{1}(x)-1}+|u_{1}|^{p_{1}(x)-1}$ belongs to 
$L^{p_{1}(x)/p_{1}(x)-1}(\Omega ).$ The dominated convergence theorem
implies that 
\begin{equation*}
\lim_{n\rightarrow +\infty }\rho _{\frac{p_{1}(x)}{p_{1}(x)-1}}\left(
|u_{1,n}|^{p_{1}(x)-1}-|u_{1}|^{p_{1}(x)-1}\right) =0.
\end{equation*}%
which by \cite[Theorem 1.4]{FZ2} shows that 
\begin{equation*}
|u_{1,n}|^{p_{1}(x)-1}\rightarrow |u_{1}|^{p_{1}(x)-1}\text{ in }L^{\frac{%
p_{1}(x)}{p_{1}(x)-1}}(\Omega ).
\end{equation*}%
Hence, bearing in mind (\ref{48}) and (\ref{47}), we derive that%
\begin{equation}
(u_{1,n}^{+})^{p_{1}(x)-1}\rightarrow (u_{1}^{+})^{p_{1}(x)-1}\text{ in }L^{%
\frac{p_{1}(x)}{p_{1}(x)-1}}(\Omega ).  \label{46}
\end{equation}%
Gathering (\ref{35}) and (\ref{46}) together, we conclude that the homotopy $%
\mathcal{H}_{1}$ is continuous from $W_{0}^{1,p_{1}(x)}(\Omega )\times
W_{0}^{1,p_{2}(x)}(\Omega )$ to $L^{\frac{p_{1}(x)}{p_{1}(x)-1}}(\Omega ).$
We proceed analogously to prove that the homotopy $\mathcal{H}_{2}$ is
continuous from $W_{0}^{1,p_{1}(x)}(\Omega )\times W_{0}^{1,p_{2}(x)}(\Omega
)$ to $L^{\frac{p_{2}(x)}{p_{2}(x)-1}}(\Omega )$.

Finally, from the estimate (\ref{32}) and the compactness of the embedding $%
W_{0}^{1,p_{i}(x)}(\Omega )\hookrightarrow L^{p_{i}(x)}(\Omega ),$ it is
readly seen that homotopies $\mathcal{H}_{1}$ and $\mathcal{H}_{2}$ are
compact. This completes the proof.
\end{proof}

\begin{proposition}
\label{P1} Assume $(\mathrm{H}^{\prime }.1)$ and $(\mathrm{H}.3)$ hold. If $%
R>0$ is sufficiently large, then the Leray-Schauder topological degree 
\begin{equation*}
\deg (\mathcal{H}(t,\cdot ,\cdot ),\mathcal{B}_{R},0)
\end{equation*}%
is well defined for every $t\in \lbrack 0,1]$. Moreover, it holds 
\begin{equation}
\begin{array}{c}
\deg (\mathcal{H}(1,\cdot ,\cdot ),\mathcal{B}_{R},0)=\deg (\mathcal{H}%
(0,\cdot ,\cdot ),\mathcal{B}_{R},0)=0.%
\end{array}
\label{21}
\end{equation}
\end{proposition}

\begin{proof}
We claim that the solution set of problem $(\mathrm{P}_{t})$ is uniformly
bounded in $W_{0}^{1,p_{1}(x)}(\Omega )\times W_{0}^{1,p_{2}(x)}(\Omega )$
with respect to $t\in \lbrack 0,1]$. To do so, suppose by contradiction that
for every positive integer $n$ there exist $t_{n}\in \lbrack 0,1]$ and a
solution $(u_{1,n},u_{2,n})$ of $(\mathrm{P}_{t_{n}})$ such that $%
t_{n}\rightarrow t\in \lbrack 0,1]$ and $\Vert (u_{1,n},u_{2,n})\Vert _{%
\mathcal{M}_{p}}\rightarrow \infty $ as $n\rightarrow \infty $. We have%
\begin{equation}
\left\{ 
\begin{array}{l}
\int_{\Omega }|\nabla u_{1,n}|^{p_{1}(x)-2}\nabla u_{1,n}\nabla \varphi
_{1}\,\mathrm{d}x=\int_{\Omega }f_{1,t_{n}}(x,u_{1,n},u_{2,n})\varphi _{1}%
\text{ }\mathrm{d}x \\ 
\int_{\Omega }|\nabla u_{2,n}|^{p_{2}(x)-2}\nabla u_{2,n}\nabla \varphi
_{2}\,\mathrm{d}x=\int_{\Omega }f_{2,t_{n}}(x,u_{1,n},u_{2,n})\varphi _{2}%
\text{ }\mathrm{d}x,%
\end{array}%
\right.   \label{10}
\end{equation}%
for all $\varphi _{i}\in W_{0}^{1,p_{i}(x)}(\Omega )$. Without loss of
generality we may admit that 
\begin{equation}
\begin{array}{c}
\theta _{n}:=\Vert u_{1,n}\Vert \rightarrow \infty \text{ as }n\rightarrow
\infty .%
\end{array}
\label{11}
\end{equation}%
Denote 
\begin{equation}
\hat{u}_{1,n}:=\frac{1}{\theta _{n}}u_{1,n}\in W_{0}^{1,p_{1}(x)}(\Omega ).
\label{11*}
\end{equation}%
Then, there exists $\hat{u}_{1}\in W_{0}^{1,p_{1}(x)}(\Omega )$ such that $%
\hat{u}_{1,n}\rightarrow \hat{u}_{1}$ weakly in $W_{0}^{1,p_{1}(x)}(\Omega )$%
, strongly in $L^{p_{1}(x)}(\Omega )$ and $a.e.$ in $\Omega .$ Putting $%
\varphi _{1}=\hat{u}_{1,n}-\hat{u}_{1}$, we have 
\begin{equation*}
\begin{array}{l}
\int_{\Omega }|\nabla \hat{u}_{1,n}|^{p_{1}(x)-2}\nabla \hat{u}_{1,n}\nabla
\varphi _{1}\,\mathrm{d}x=\int_{\Omega }\frac{1}{\theta _{n}^{p_{1}(x)-1}}%
|\nabla u_{1,n}|^{p_{1}(x)-2}\nabla u_{1,n}\nabla \varphi _{1}\,\mathrm{d}x
\\ 
=\int_{\Omega }\frac{1}{\theta _{n}^{p_{1}(x)-1}}|\nabla
u_{1,n}|^{p_{1}(x)-2}(\nabla u_{1,n}\nabla \varphi _{1}\chi _{\{\varphi
_{1}\geq 0\}}+\nabla u_{1,n}\nabla \varphi _{1}\chi _{\{\varphi _{1}<0\}})\,%
\mathrm{d}x.%
\end{array}%
\end{equation*}%
Noticing that 
\begin{equation*}
\begin{array}{l}
\int_{\Omega }\frac{1}{\theta _{n}^{p_{1}(x)-1}}|\nabla
u_{1,n}|^{p_{1}(x)-2}\nabla u_{1,n}\nabla \varphi _{1}\chi _{\{\varphi
_{1}\geq 0\}}\,\mathrm{d}x \\ 
=\int_{\Omega }\frac{1}{\theta _{n}^{p_{1}(x)-1}}|\nabla
u_{1,n}|^{p_{1}(x)-2}\nabla u_{1,n}\nabla \left( \varphi _{1}\chi
_{\{\varphi _{1}\geq 0\}}\right) \,\mathrm{d}x,%
\end{array}%
\end{equation*}%
\begin{equation*}
\begin{array}{l}
\int_{\Omega }\frac{1}{\theta _{n}^{p_{1}(x)-1}}|\nabla
u_{1,n}|^{p_{1}(x)-2}\nabla u_{1,n}\nabla \varphi _{1}\chi _{\{\varphi
_{1}<0\}}\,\mathrm{d}x \\ 
=\int_{\Omega }\frac{1}{\theta _{n}^{p_{1}(x)-1}}|\nabla
u_{1,n}|^{p_{1}(x)-2}\nabla u_{1,n}\nabla \left( \varphi _{1}\chi
_{\{\varphi _{1}<0\}}\right) \,\mathrm{d}x,%
\end{array}%
\end{equation*}%
by $(\mathrm{H}^{\prime }.1)(\mathrm{i})$, the successive application of
Corollary \ref{C} for $\varphi =\varphi _{1}\chi _{\{\varphi _{1}\geq 0\}}$
and $\varphi =-\varphi _{1}\chi _{\{\varphi _{1}<0\}}$ guarantee the
existence of $x_{0},\hat{x}_{0}\in \Omega $ such that 
\begin{equation*}
\begin{array}{l}
\int_{\Omega }\frac{1}{\theta _{n}^{p_{1}(x)-1}}|\nabla
u_{1,n}|^{p_{1}(x)-2}\nabla u_{1,n}\nabla \left( \varphi _{1}\chi
_{\{\varphi _{1}\geq 0\}}\right)  \\ 
=\frac{1}{\theta _{n}^{p_{1}(x_{0})-1}}\int_{\Omega }|\nabla
u_{1,n}|^{p_{1}(x)-2}\nabla u_{1,n}\nabla \left( \varphi _{1}\chi
_{\{\varphi _{1}\geq 0\}}\right) 
\end{array}%
\end{equation*}%
and%
\begin{equation*}
\begin{array}{l}
-\int_{\Omega }\frac{1}{\theta _{n}^{p_{1}(x)-1}}|\nabla
u_{1,n}|^{p_{1}(x)-2}\nabla u_{1,n}\nabla \left( -\varphi _{1}\chi
_{\{\varphi _{1}<0\}}\right)  \\ 
=-\frac{1}{\theta _{n}^{p_{1}(\hat{x}_{0})-1}}\int_{\Omega }|\nabla
u_{1,n}|^{p_{1}(x)-2}\nabla u_{1,n}\nabla \left( -\varphi _{1}\chi
_{\{\varphi _{1}<0\}}\right) .%
\end{array}%
\end{equation*}%
Thence%
\begin{equation*}
\begin{array}{l}
\int_{\Omega }|\nabla \hat{u}_{1,n}|^{p_{1}(x)-2}\nabla \hat{u}_{1,n}\nabla
\varphi _{1}\,\mathrm{d}x \\ 
=\frac{1}{\theta _{n}^{p_{1}(x_{0})-1}}\int_{\Omega }|\nabla
u_{1,n}|^{p_{1}(x)-2}\nabla u_{1,n}\nabla \varphi _{1}\chi _{\{\varphi
_{1}\geq 0\}} \\ 
\text{ \ \ }+\frac{1}{\theta _{n}^{p_{1}(\hat{x}_{0})-1}}\int_{\Omega
}|\nabla u_{1,n}|^{p_{1}(x)-2}\nabla u_{1,n}\nabla \varphi _{1}\chi
_{\{\varphi _{1}<0\}} \\ 
=(\frac{1}{\theta _{n}^{p_{1}(x_{0})-1}}+\frac{1}{\theta _{n}^{p_{1}(\hat{x}%
_{0})-1}})\int_{\Omega }|\nabla u_{1,n}|^{p_{1}(x)-2}\nabla u_{1,n}\nabla
\varphi _{1} \\ 
=(\frac{1}{\theta _{n}^{p_{1}(x_{0})-1}}+\frac{1}{\theta _{n}^{p_{1}(\hat{x}%
_{0})-1}})\int_{\Omega }f_{1,t_{n}}(x,u_{1,n},u_{2,n})\varphi _{1}\text{ }%
\mathrm{d}x,%
\end{array}%
\end{equation*}%
which, by (\ref{12}), is equivalent to%
\begin{equation}
\begin{array}{l}
\int_{\Omega }|\nabla \hat{u}_{1,n}|^{p_{1}(x)-2}\nabla \hat{u}_{1,n}\nabla (%
\hat{u}_{1,n}-\hat{u}_{1})\,\mathrm{d}x \\ 
=(\frac{1}{\theta _{n}^{p_{1}(x_{0})-1}}+\frac{1}{\theta _{n}^{p_{1}(\hat{x}%
_{0})-1}})\left[ \int_{\Omega }t_{n}f_{1}(x,u_{1,n},u_{2,n})(\hat{u}_{1,n}-%
\hat{u}_{1})\text{ }\mathrm{d}x\right.  \\ 
\left. +(1-t_{n})\int_{\Omega }\left( J_{1}(\frac{u_{1,n}^{+}}{\max
\{1,\Vert u_{1,n}\Vert \}})^{p_{1}(x)-1}+\delta \lambda _{1,p_{1}}\phi
_{1,p_{1}}^{p_{1}(x)-1}\right) (\hat{u}_{1,n}-\hat{u}_{1})\text{ }\mathrm{d}x%
\right] .%
\end{array}
\label{13}
\end{equation}%
Thus, bearing in mind $(\mathrm{H}^{\prime }.1)(\mathrm{i})$ and (\ref{11}),
one gets%
\begin{equation*}
\begin{array}{l}
\left\vert \int_{\Omega }|\nabla \hat{u}_{1,n}|^{p_{1}(x)-2}\nabla \hat{u}%
_{1,n}\nabla (\hat{u}_{1,n}-\hat{u}_{1})\,\mathrm{d}x\right\vert  \\ 
\leq \frac{2}{\theta _{n}^{p_{1}^{-}-1}}\left[ \int_{\Omega
}t_{n}f_{1}(x,u_{1,n},u_{2,n})\left\vert \hat{u}_{1,n}-\hat{u}%
_{1}\right\vert \text{ }\mathrm{d}x\right.  \\ 
\text{ \ \ \ \ \ \ \ \ \ \ \ }\left. +(1-t_{n})\int_{\Omega }\left( J_{1}(%
\frac{u_{1,n}^{+}}{\max \{1,\Vert u_{1,n}\Vert \}})^{p_{1}(x)-1}+\delta
\lambda _{1,p_{1}}\phi _{1,p_{1}}^{p_{1}(x)-1}\right) \left\vert \hat{u}%
_{1,n}-\hat{u}_{1}\right\vert \text{ }\mathrm{d}x\right] ,%
\end{array}%
\end{equation*}%
as well as%
\begin{equation}
\begin{array}{l}
\frac{2J_{1}}{\theta _{n}^{p_{1}^{-}-1}}\int_{\Omega }\frac{%
(u_{1,n}^{+})^{p_{1}(x)-1}}{(\max \{1,\Vert u_{1,n}\Vert \})^{p_{1}(x)-1}}|%
\hat{u}_{1,n}-\hat{u}_{1}|\,\mathrm{d}x \\ 
=\frac{2J_{1}}{\theta _{n}^{p_{1}^{-}-1}}\int_{\Omega }\frac{%
(u_{1,n}^{+})^{p_{1}(x)-1}}{\max \{1,\theta _{n}\}^{p_{1}(x)-1}}|\hat{u}%
_{1,n}-\hat{u}_{1}|\,\mathrm{d}x \\ 
=\frac{2J_{1}}{\theta _{n}^{p_{1}^{-}-1}}\int_{\Omega }(\hat{u}%
_{1,n}^{+})^{p_{1}(x)-1}|\hat{u}_{1,n}-\hat{u}_{1}|\,\mathrm{d}x\leq
\int_{\Omega }(\hat{u}_{1,n}^{+})^{p_{1}(x)-1}|\hat{u}_{1,n}-\hat{u}_{1}|\,%
\mathrm{d}x%
\end{array}%
\end{equation}%
and%
\begin{equation*}
\begin{array}{l}
\frac{2\delta \lambda _{1,p_{1}}}{\theta _{n}^{p_{1}^{-}-1}}\int_{\Omega
}\phi _{1,p_{1}}^{p_{1}(x)-1}|\hat{u}_{1,n}-\hat{u}_{1}|\,\mathrm{d}x\leq 
\frac{2\delta \lambda _{1,p_{1}}}{\theta _{n}^{p_{1}^{-}-1}}\int_{\Omega
}\left\Vert \phi _{1,p_{1}}\right\Vert _{\infty }^{p_{1}(x)-1}|\hat{u}_{1,n}-%
\hat{u}_{1}|\,\mathrm{d}x \\ 
\leq \frac{2\delta \lambda _{1,p_{1}}}{\theta _{n}^{p_{1}^{-}-1}}\max
\{1,\left\Vert \phi _{1,p_{1}}\right\Vert _{\infty
}\}^{p_{1}^{+}-1}\int_{\Omega }|\hat{u}_{1,n}-\hat{u}_{1}|\,\mathrm{d}x\leq
\int_{\Omega }|\hat{u}_{1,n}-\hat{u}_{1}|\,\mathrm{d}x.%
\end{array}%
\end{equation*}%
On the other hand, assumption $(\mathrm{H.3})$ yields $\bar{\eta}>0$ and $%
\kappa =\kappa (\bar{\eta})>0$ fulfilling%
\begin{equation}
|s_{1}|>\kappa \Longrightarrow |f_{1}(x,s_{1},s_{2})|<\bar{\eta}%
|s_{1}|^{p_{1}^{-}-1}\text{ for }x\in \Omega ,\text{ }s_{2}\in 
\mathbb{R}
.  \label{14}
\end{equation}%
Given $n\in 
\mathbb{N}
$ observe that%
\begin{equation*}
\begin{array}{l}
\frac{2t_{n}}{\theta _{n}^{p_{1}^{-}-1}}\int_{\Omega
}f_{1}(x,u_{1,n},u_{2,n})|\hat{u}_{1,n}-\hat{u}_{1}|\,\mathrm{d}x \\ 
=\frac{2t_{n}}{\theta _{n}^{p_{1}^{-}-1}}\left[ \int_{|u_{n}|>\kappa
}f_{1}(x,u_{1,n},u_{2,n})|\hat{u}_{1,n}-\hat{u}_{1}|\,\mathrm{d}x\right.  \\ 
\text{ \ \ \ \ \ \ \ \ \ \ \ }\left. +\int_{|u_{n}|\leq \kappa
}f_{1}(x,u_{1,n},u_{2,n})|\hat{u}_{1,n}-\hat{u}_{1}|\,\mathrm{d}x\right] .%
\end{array}%
\end{equation*}%
Thus, $($\textrm{H.}$3)$ and (\ref{14}) entail 
\begin{equation}
\begin{array}{l}
\int_{|u_{n}|>\kappa }\frac{2t_{n}}{\theta _{n}^{p_{1}^{-}-1}}%
f_{1}(x,u_{1,n},u_{2,n})|\hat{u}_{1,n}-\hat{u}_{1}|\,\mathrm{d}x \\ 
=\int_{|u_{n}|>\kappa }2t_{n}|\hat{u}_{1,n}|^{p_{1}^{-}-1}\frac{%
f_{1}(x,\theta _{n}\hat{u}_{n},v_{n})}{(\theta _{n}|\hat{u}%
_{n}|)^{p_{1}^{-}-1}}|\hat{u}_{1,n}-\hat{u}_{1}|\,\mathrm{d}x \\ 
\leq 2\bar{\eta}\int_{|u_{n}|>\kappa }|\hat{u}_{1,n}|^{p_{1}^{-}-1}|\hat{u}%
_{1,n}-\hat{u}_{1}|\,\mathrm{d}x,%
\end{array}
\label{16}
\end{equation}%
while, by $(\mathrm{H}^{\prime }.1)(\mathrm{ii}$) and (\ref{11}), we have%
\begin{equation}
\begin{array}{l}
\int_{|u_{n}|\leq \kappa }|\frac{2t_{n}}{\theta _{n}^{p_{1}^{-}-1}}%
f_{1}(x,u_{1,n},u_{2,n})||\hat{u}_{1,n}-\hat{u}_{1}|\,\mathrm{d}x\leq
M\int_{|u_{n}|\leq \kappa }|\hat{u}_{1,n}-\hat{u}_{1}|\,\mathrm{d}x \\ 
\leq M\int_{\Omega }|\hat{u}_{1,n}-\hat{u}_{1}|\,\mathrm{d}x,%
\end{array}
\label{18}
\end{equation}%
Thus, passing to the limit as $n\rightarrow \infty $, Lebesgue dominate
convergence theorem implies 
\begin{equation}
\lim_{n\rightarrow \infty }\left\langle -\Delta _{p_{1}(x)}\hat{u}_{1,n},%
\hat{u}_{1,n}-\hat{u}_{1}\right\rangle =0.  \label{17}
\end{equation}%
Consequently, the $S_{+}$ property of the operator $-\Delta _{p_{1}(x)}$
shows that%
\begin{equation*}
\hat{u}_{1,n}\rightarrow \hat{u}_{1}\text{ strongly in }W_{0}^{1,p_{1}(x)}(%
\Omega )\text{ with }\Vert \hat{u}_{1}\Vert =1.
\end{equation*}%
Acting in (\ref{13}) with $\varphi _{1}=\hat{u}_{1}$ instead of $\varphi
_{1}=\hat{u}_{1,n}-\hat{u}_{1}$ and passing to the limit as $n\rightarrow
\infty $ one gets%
\begin{equation}
\int_{\Omega }|\nabla \hat{u}|^{p_{1}(x)}\,\mathrm{d}x\leq
(1-t)J_{1}\int_{\Omega }(\hat{u}_{1}^{+})^{p_{1}(x)}\,\mathrm{d}x,\text{ for 
}t\in \lbrack 0,1].  \label{19}
\end{equation}%
Testing with $-\hat{u}_{1}^{-}$ in (\ref{13}), using $(\mathrm{H}^{\prime
}.1)\mathrm{(i)}$ and passing to the limit leads to $\hat{u}_{1}=\hat{u}%
_{1}^{+}$, which is nonzero because $\Vert \hat{u}_{1}\Vert =1$. Thus%
\begin{equation}
\int_{\Omega }|\nabla \hat{u}_{1}|^{p_{1}(x)}\,\mathrm{d}x\leq
(1-t)J_{1}\int_{\Omega }\hat{u}_{1}^{p_{1}(x)}\,\mathrm{d}x,\text{ for }t\in
\lbrack 0,1].  \label{20}
\end{equation}%
If $t=1$ then $\hat{u}_{1}=0$ which contradicts the fact that $\hat{u}\neq 0.
$ Assume $t\in \lbrack 0,1).$ By (\ref{71}) and (\ref{20}) it follows that 
\begin{equation*}
(\lambda _{1,p_{1}}-(1-t)J_{1}))\int_{\Omega }\hat{u}_{1}^{p_{1}(x)}\,%
\mathrm{d}x\leq 0,
\end{equation*}%
which is a contradiction because $(1-t)J_{1}<\lambda _{1,p_{1}}$ for $t\in
\lbrack 0,1]$ (see (\ref{43})) and $\hat{u}_{1}>0.$ The claim is thus proved.

As a consequence of the previous claim, the Leray-Schauder topological
degree $\deg (\mathcal{H}(t,\cdot ,\cdot ),\mathcal{B}_{R},0)$ is well
defined for every $t\in \lbrack 0,1]$.

The task is now to prove (\ref{21}). Thanks to the homotopy invariance
property of the Leray-Schauder topological degree, the first equality in (%
\ref{21}) is fulfilled. For $t=0$, $(\mathrm{P}_{0})$ is expressed as a
decoupled system:%
\begin{equation*}
(\mathrm{P}_{0})\qquad \left\{ 
\begin{array}{ll}
-\Delta _{p_{i}(x)}u_{i}=J_{i}\frac{(u_{i}^{+})^{p_{i}(x)-1}}{(\max
\{1,\Vert u_{i}\Vert \})^{p_{i}(x)-1}}+\delta \lambda _{1,p_{i}}\phi
_{1,p_{i}(x)}^{p_{i}(x)-1} & \text{in }\Omega \\ 
u_{i}=0 & \text{on }\partial \Omega ,%
\end{array}%
\right.
\end{equation*}%
which, by Lemma \ref{L9}, has no solutions. Thus, the second equality in (%
\ref{21}) holds true. This completes the proof.
\end{proof}

\mathstrut

\subsection{Topological degree on $\mathcal{B}_{\tilde{R}}$}

We slightly modify the homotopy $\mathcal{H}$ related to problem $(\mathrm{P}%
_{t})$. Specifically, let us consider for every $t\in \lbrack 0,1]$ the
Dirichlet problem:%
\begin{equation*}
(\mathrm{\tilde{P}}_{t})\qquad \left\{ 
\begin{array}{ll}
-\Delta _{p_{i}(x)}u_{i}=\tilde{f}_{i,t}(x,u,v) & \text{in }\Omega \\ 
u_{i}=0 & \text{on }\partial \Omega ,%
\end{array}%
\right.
\end{equation*}%
with 
\begin{equation}
\begin{array}{l}
\tilde{f}_{i,t}(x,u_{1},u_{2})=tf_{i}(x,u_{1},u_{2})+(1-t)J_{i}\frac{%
(u_{i}^{+})^{p_{i}(x)-1}}{(\max \{1,\Vert u_{i}\Vert \})^{p_{i}(x)-1}},%
\end{array}%
\end{equation}%
where $J_{i}$ satisfies (\ref{43}).

For a constant $\tilde{R}>0$, let define the homotopy 
\begin{equation*}
\begin{array}{lll}
\mathcal{\tilde{H}}: & [0,1]\times \overline{\mathcal{B}_{R}} & \rightarrow
W^{-1,p_{1}^{\prime }(x)}(\Omega )\times W^{-1,p_{2}^{\prime }(x)}(\Omega )
\\ 
& (t,u_{1},u_{2}) & \rightarrow (\mathcal{\tilde{H}}_{1}(t,u_{1},u_{2}),%
\mathcal{\tilde{H}}_{2}(t,u_{1},u_{2}))%
\end{array}%
\end{equation*}%
where $\mathcal{\tilde{H}}_{i}$ are given by 
\begin{equation*}
\begin{array}{l}
\left\langle \mathcal{\tilde{H}}_{i}(t,u_{1},u_{2}),\varphi
_{i}\right\rangle =\int_{\Omega }|\nabla u_{i}|^{p_{i}(x)-2}\nabla
u_{i}\nabla \varphi _{i}\,\mathrm{d}x-\int_{\Omega }\tilde{f}%
_{i,t}(x,u_{1},u_{2})\varphi _{i}\text{ }\mathrm{d}x,%
\end{array}%
\end{equation*}%
for $\varphi _{i}\in W_{0}^{1,p_{i}(x)}(\Omega ),$ and $\overline{\mathcal{B}%
_{\tilde{R}}}$ is the closure of $\mathcal{B}_{\tilde{R}}$ in $%
W_{0}^{1,p_{1}(x)}(\Omega )\times W_{0}^{1,p_{2}(x)}(\Omega )$ with%
\begin{equation*}
\mathcal{B}_{\tilde{R}}:=\left\{ (u_{1},u_{2})\in W_{0}^{1,p_{1}(x)}(\Omega
)\times W_{0}^{1,p_{2}(x)}(\Omega ):\,\Vert (u_{1},u_{2})\Vert <\tilde{R}%
\right\}
\end{equation*}

\begin{proposition}
\label{P2} Assume that condition $(\mathrm{H}^{\prime }.1)$ and $(\mathrm{H}%
.3)$ are satisfied. If $\tilde{R}>0$ is sufficiently large, then the
Leray-Schauder topological degree 
\begin{equation*}
\deg (\mathcal{\tilde{H}}(t,\cdot ,\cdot ),\mathcal{B}_{\tilde{R}},0)
\end{equation*}%
is well defined for every $t\in \lbrack 0,1]$. Moreover, it holds 
\begin{equation}
\begin{array}{c}
\deg (\mathcal{\tilde{H}}(1,\cdot ,\cdot ),\mathcal{B}_{\tilde{R}},0)=\deg (%
\mathcal{\tilde{H}}(0,\cdot ,\cdot ),\mathcal{B}_{\tilde{R}},0)=1.%
\end{array}
\label{33}
\end{equation}
\end{proposition}

\begin{proof}
Arguing as in the proof of Proposition \ref{P1} we show that the solution
set of problem $(\mathrm{\tilde{P}}_{t})$ is bounded in $W_{0}^{1,p_{1}(x)}(%
\Omega )\times W_{0}^{1,p_{2}(x)}(\Omega )$ uniformly with respect to $t\in
\lbrack 0,1]$. Thus, for $\tilde{R}>0$ is sufficiently large the
Leray-Schauder topological degree $\deg (\mathcal{\tilde{H}}(t,\cdot ,\cdot
),\mathcal{B}_{\tilde{R}},0)$ is well defined for every $t\in \lbrack 0,1]$.
Moreover, the first equality in (\ref{33}) is true thanks to the homotopy
invariance property of Leray-Schauder topological degree.

On the other hand, for $t=0$, $(\mathrm{\tilde{P}}_{0})$ is expressed as a
decoupled system:%
\begin{equation*}
(\mathrm{\tilde{P}}_{0})\qquad \left\{ 
\begin{array}{ll}
-\Delta _{p_{1}(x)}u=J_{1}\frac{(u^{+})^{p_{1}(x)-1}}{(\max \{1,\Vert u\Vert
\})^{p_{1}(x)-1}} & \text{in }\Omega \\ 
-\Delta _{p_{2}(x)}v=J_{2}\frac{(v^{+})^{p_{2}(x)-1}}{(\max \{1,\Vert v\Vert
\})^{p_{2}(x)-1}} & \text{in }\Omega \\ 
u,v=0 & \text{on }\partial \Omega ,%
\end{array}%
\right.
\end{equation*}%
which, since $J_{i}\in (0,\lambda _{1,p}(p^{-}-1))$, admits only the trivial
solution $(u,v)=(0,0)$. Then, from the definition of Leray-Schauder
topological degree together with its homotopy invariance property, the
equalities in (\ref{33}) hold true. This completes the proof.
\end{proof}

\subsection{Topological degree on $\mathcal{B}_{R}\backslash \overline{%
\mathcal{B}_{\hat{R}}}$}

Fix $\hat{R}>0$ in Proposition \ref{P2} so large that every element $%
(u_{1},u_{2})\ $in $[-\overline{u}_{1},\overline{u}_{1}]\times \lbrack -%
\overline{u}_{2},\overline{u}_{2}]$ belongs to $\mathcal{B}_{\hat{R}}$. Take 
$R>\hat{R}$, with $R$ so large to fulfill the conclusion of Proposition \ref%
{P1}. For this construction, it is essential to observe that $\hat{R}>0$ in
Proposition \ref{P2} and $R>0$ in Proposition \ref{P1} must necessarily
verify $\hat{R}<R$. This is the consequence of the weak comparison principle
in Lemma \ref{L6} applied to problems $(\mathrm{P}_{t})$ and $(\mathrm{%
\tilde{P}}_{t})$ making use of the inequality $\tilde{f}%
_{i,t}(x,s_{1},s_{2})<f_{i,t}(x,s_{1},s_{2}),$ for a.e. $x\in \Omega ,$ all $%
s_{1},s_{2}\in \mathbb{R},$ $t\in \lbrack 0,1).$ Hence, the strict inclusion 
$\overline{\mathcal{B}_{\hat{R}}}\subset \mathcal{B}_{R}$ is fulfilled.

In view of the expressions of the homotopies $\mathcal{H}$ and $\mathcal{%
\tilde{H}}$ used in Propositions \ref{P1} and \ref{P2}, it is seen that 
\begin{equation}
\mathcal{H}(1,\cdot ,\cdot )=\mathcal{\tilde{H}}(1,\cdot ,\cdot )\,\,\text{in%
}\ \overline{\mathcal{B}_{\hat{R}}}.  \label{50}
\end{equation}%
The Leray-Schauder degree $\deg (\mathcal{H}(1,\cdot ,\cdot ),\mathcal{B}%
_{R}\backslash \partial \mathcal{B}_{\hat{R}},0)$ of $\mathcal{H}(1,\cdot
,\cdot )$ on $\mathcal{B}_{R}\backslash \overline{\mathcal{B}_{\hat{R}}}$
makes sense according to (\ref{50}) because it was shown in Propositions \ref%
{P1} and \ref{P2} that $\mathcal{H}(1,\cdot ,\cdot )$ and $\mathcal{\tilde{H}%
}(1,\cdot ,\cdot )$ do not vanish on $\partial \mathcal{B}_{R}$ and $%
\partial \mathcal{B}_{\hat{R}}$, respectively. Then the excision property of
Leray-Schauder degree (see, e.g., \cite[p. 72]{MMP}) yields 
\begin{equation*}
\begin{array}{c}
\deg (\mathcal{H}(1,\cdot ,\cdot ),\mathcal{B}_{R},0)=\deg (\mathcal{H}%
(1,\cdot ,\cdot ),\mathcal{B}_{R}\backslash \partial \mathcal{B}_{\hat{R}%
},0),%
\end{array}%
\end{equation*}%
whereas by virtue of the domain additivity property of Leray-Schauder degree
it turns out that 
\begin{equation*}
\begin{array}{c}
\deg (\mathcal{H}(1,\cdot ,\cdot ),\mathcal{B}_{R},0)=\deg (\mathcal{H}%
(1,\cdot ,\cdot ),\mathcal{B}_{\hat{R}},0)+\deg (\mathcal{H}(1,\cdot ,\cdot
),\mathcal{B}_{R}\backslash \overline{\mathcal{B}_{\hat{R}}},0).%
\end{array}%
\end{equation*}%
Combining the preceding equalities with (\ref{21}) and (\ref{33}), we infer
that 
\begin{equation*}
\begin{array}{c}
\deg (\mathcal{H}(1,\cdot ,\cdot ),\mathcal{B}_{R}\backslash \overline{%
\mathcal{B}_{\hat{R}}},0)=-1.%
\end{array}%
\end{equation*}%
Therefore, there exists $(\breve{u}_{1},\breve{u}_{2})\in \mathcal{B}%
_{R}\backslash \overline{\mathcal{B}_{\hat{R}}}$ satisfying $\mathcal{H}(1,%
\breve{u}_{1},\breve{u}_{2})=0$. This implies that the pair $(\breve{u}_{1},%
\breve{u}_{2})$ is a solution of system $(\mathrm{P})$ belonging to the set $%
\mathcal{B}_{R}\backslash \overline{\mathcal{B}_{\hat{R}}}$.

\subsection{Proof of Theorem \protect\ref{T2}}

Since $(\breve{u}_{1},\breve{u}_{2})\in \mathcal{B}_{R}\backslash \overline{%
\mathcal{B}_{\hat{R}}}$ and the ordered rectangle $[-\overline{u}_{1},%
\overline{u}_{1}]\times \lbrack -\overline{u}_{2},\overline{u}_{2}]$ is
contained in the ball $\mathcal{B}_{\hat{R}}$, we have that $(\breve{u}_{1},%
\breve{u}_{2})\not\in \lbrack -\overline{u}_{1},\overline{u}_{1}]\times
\lbrack -\overline{u}_{2},\overline{u}_{2}]$. In particular, we note that $(%
\breve{u}_{1},\breve{u}_{2})\not=(u_{1,+},u_{2,+})$, so $(\breve{u}_{1},%
\breve{u}_{2})$ is a second nontrivial positive solution of system $(\mathrm{%
P})$. This completes the proof.

\begin{acknowledgement}
A. Moussaoui was supported by the Directorate-General of Scientific Research
and Technological Development (DGRSDT).
\end{acknowledgement}


\begin{thebibliography}{99}
\bibitem{AM1} E. Acerbi \& G. Mingione, \emph{Regularity results for
stationary electrorheological fluids}, Arch. Ration. Mech. Anal. 164 (2002),
213--259.

\bibitem{AM} C. O. Alves \& A. Moussaoui, \emph{Existence and regularity of
solutions for a class of singular (}$\emph{p(x),q(x)}$\emph{)- Laplacian
system}, Complex Var. Elliptic Eqts. 63 (2) (2017), 188-210.

\bibitem{AMT} Alves, A. Moussaoui \& L. Tavares, \emph{An elliptic system
with logarithmic nonlinearity}, Adv. Nonlinear Anal. 8 (2019), 928-945.

\bibitem{B} D. Banks, \emph{An integral inequality}, Proceedings Amer. Math.
Soc. 5 (14) (1963), 823-828.

\bibitem{CLR} Y. Chen, S. Levine \& M. Rao, \emph{Variable exponent, linear
growth functionals in image restoration}. SIAM J. Appl. Math. 66 (2006),
1383--1406.

\bibitem{DM} H. Didi \& A. Moussaoui, \emph{Multiple positive solutions for
a class of quasilinear singular elliptic systems}, Rendiconti del Circolo
Matematico di Palermo 69 (2020), 977--994.

\bibitem{DKM} H. Didi, B. Khodja \& A. Moussaoui, \emph{Singular quasilinear
elliptic systems with (super-) homogeneous condition}, J. Sib. Fed. Univ.
Math. Phys. 13 (2) (2020), 151-159.

\bibitem{DHHR} L. Diening, P. Harjulehto, P. H\"{a}st\"{o} \& M. Ru\v{z}%
icka, \emph{Lebesgue and Sobolev Spaces with Variable Exponents},
Springer-Verlag, Berlin (2011).

\bibitem{FZZ1} X. Fan, Q. Zhang \& D. Zhao, \emph{Eigenvalues of }$p(x)$%
\emph{-Laplacian Dirichlet problem}, J. Math. Anal. Appl. 302(2005), 306-317.

\bibitem{F} X. Fan, \emph{Global }$C^{1,\alpha }$\emph{\ regularity for
variable exponent elliptic equations in divergence form}, J. Diff. Eqts. 235
(2007), 397-417.

\bibitem{FZ} X. Fan \& D. Zhao, \emph{A class of De Giorgi type and H\"{o}%
lder continuity}, Nonlinear Anal. 36 (1996) 295--318.

\bibitem{FZZ2} X. Fan, D. Zhao \& Q. Zhang, \emph{Astrong maximum principle
for }$p(x)$\emph{-Laplace equations}, Chin. J. Contemp. Math. 21(1) (2000),
1--7.

\bibitem{FZ2} X. Fan \& D. Zhao, \emph{On the Spaces }$L^{p(x)}(\Omega )$%
\emph{\ and }$W_{0}^{1,p(x)}(\Omega )$, J. Math. Anal. App. 263 (2001),
424-446.

\bibitem{KM} B. Khodja \& A. Moussaoui, \emph{Positive\ solutions\ for\
infinite\ semipositone/positone\ quasilinear\ elliptic\ systems\ with\
singular\ and\ superlinear\ terms}, Diff. Eqts. App. 8(4) (2016), 535-546.

\bibitem{MMP} D. Motreanu, V.V. Motreanu \& N. Papageorgiou, \emph{%
Topological and variational methods with applications to nonlinear boundary
value problems}. Springer, New York, 2014.

\bibitem{MMP2} D. Motreanu, A. Moussaoui \& D. S. Pereira, \emph{Multiple
solutions for nonvariational quasilinear elliptic systems}, Mediterranean J.
Math. (2018), doi : 10.1007/s00009-018-1133-9.

\bibitem{MM} D. Motreanu \& A. Moussaoui, \emph{An existence result for a
class of quasilinear singular competitive elliptic systems}, Applied Math.
Letters 38 (2014), 33-37.

\bibitem{MM2} D. Motreanu \& A. Moussaoui, \emph{Existence and boundedness
of solutions for a singular cooperative quasilinear elliptic system},
Complex Var. Elliptic Eqts. 59 (2014), 285-296.

\bibitem{MM3} D. Motreanu \& A. Moussaoui, $\emph{A}$ $\emph{quasilinear}$ $%
\emph{singular}$ $\emph{elliptic}$ $\emph{system}$ $\emph{without}$ $\emph{%
cooperative}$ $\emph{structure}$, Acta Math. Sci.\emph{\ }34 (B) (2014),
905-916.

\bibitem{MV1} A. Moussaoui \& J. V\'{e}lin, \emph{Existence and boundedness
of solutions for systems of quasilinear elliptic equations}, Acta Math.
Scientia 41 (2021), 397-412.

\bibitem{MV2} A. Moussaoui \& J. V\'{e}lin, \emph{Existence and a priori
estimates of solutions for quasilinear singular elliptic systems with
variable exponents}, J. Elliptic Parabolic Eqts. (2018), doi :
10.1007/s41808-018-0025-2.

\bibitem{MV3} A. Moussaoui \& J. V\'{e}lin, \emph{On the first eigenvalue
for a (}$p(x),q(x)$\emph{)-Laplacian elliptic system}, Electron. J. Qual.
Theory Diff. Eqts. 66 (2019), 1-22.

\bibitem{R} S.H. Rasouli, \emph{On a Picone's identity for the }$A_{p(x)}$%
\emph{-Laplacian and its applications}, Bull. Iranian Math. Soc. 7 (43)
(2017), 2449-2455.

\bibitem{R} M. Ruzicka, \emph{Electrorheological fluids: Modeling and
mathematical theory}, Lecture Notes in Math., vol. 1748. Springer, Berlin
(2000)

\bibitem{Z} Q. Zhang, Existence of positive solutions for a class of $p(x)$%
-Laplacian systems, J. Math. Anal. Appl. 333 (2007), 591-603.
\end{thebibliography}
\end{document}